\documentclass[12pt,a4paper]{article}
\usepackage[latin1]{inputenc}
\usepackage[T1]{fontenc}
\usepackage{amsmath, amsthm}
\usepackage{amsfonts, bm}
\usepackage{amssymb}
\usepackage{graphicx}
\usepackage[cal=esstix,frak=euler,scr=boondox]{mathalfa}
\usepackage[dvipsnames]{xcolor}
\usepackage[shortlabels]{enumitem}
\usepackage[colorlinks = true, urlcolor = blue, pdfborder= 0 0 0.5, urlbordercolor = blue ]{hyperref}
\usepackage{algpseudocode,algorithm}

\newtheorem{theorem}{Theorem}
\newtheorem{lemma}[theorem]{Lemma}
\newtheorem{proposition}[theorem]{Proposition}
\newtheorem{corollary}[theorem]{Corollary}
\newtheorem{remark}[theorem]{Remark}
\numberwithin{theorem}{section}

\usepackage{caption}
\usepackage[skip=0.5ex]{subcaption}

\newcommand{\nocontentsline}[3]{}
\newcommand{\tocless}[2]{\bgroup\let\addcontentsline=\nocontentsline#1{#2}\egroup}

\newcommand{\C}{\mathbb{C}}
\newcommand{\N}{\mathbb{N}}
\newcommand{\R}{\mathbb{R}}
\newcommand{\idop}{\mathcal{I}}

\renewcommand{\d}{\mathrm{d}}

\renewcommand{\L}{\mathcal{L}}
\renewcommand{\P}{\mathcal{P}}
\DeclareMathOperator{\im}{im}
\DeclareMathOperator{\spann}{span}
\DeclareMathOperator{\sgn}{sign}

\newcommand{\Hd}[2][d,p]{\mathcal{H}_{#2}^{#1}}
\newcommand{\Bd}[2][p,d]{\mathcal{E}_{#2}^{#1}}

\renewcommand{\k}{\mathbf{k}}

\newcommand{\x}{\mathbf{x}}
\newcommand{\z}{\mathbf{z}}
\newcommand{\btheta}{{\boldsymbol{\theta}}}
\newcommand{\inv}{I}

\makeatletter
\def\widebreve{\mathpalette\wide@breve}
\def\wide@breve#1#2{\sbox\z@{$#1#2$}%
     \mathop{\vbox{\m@th\ialign{##\crcr
\kern0.08em\brevefill#1{0.8\wd\z@}\crcr\noalign{\nointerlineskip}%
                    $\hss#1#2\hss$\crcr}}}\limits}
\def\brevefill#1#2{$\m@th\sbox\tw@{$#1($}%
  \hss\resizebox{#2}{\wd\tw@}{\rotatebox[origin=c]{90}{\upshape$\langle$}}\hss$}
\makeatletter

\newcommand\modck{\check m}
\renewcommand\mod{\allowbreak\space\mathord{ \rm mod}\allowbreak\, }

\newcommand*\Let[2]{\State #1 $\gets$ #2}
\algrenewcommand\algorithmicrequire{\textbf{Input:}}
\algrenewcommand\algorithmicensure{\textbf{Output:}}

\algrenewcommand\alglinenumber[1]{ \sf\scriptsize{#1}}

\newcommand\apack{\mathscr A}

\date{ }

\begin{document}
\title{Hausdorff dimension of the Apollonian gasket }

\author{Polina L. Vytnova\thanks{The University of Surrey, {\tt
p.vytnova@surrey.ac.uk}}\, and Caroline L. Wormell\thanks{The University of
Sydney, {\tt caroline.wormell@sydney.edu.au}}}

\maketitle

\abstract{The Apollonian gasket is a well-studied circle packing. Important
properties of the packing, including the distribution of the circle
radii, are governed by its Hausdorff dimension. No closed form is
currently known for the Hausdorff dimension, and its computation is a
special case of a more general and hard problem: effective, rigorous estimates of
dimension of a parabolic limit set.

In this paper we develop an efficient method for solving this
problem  
which allows us to compute the dimension of the gasket to $128$ decimal places and rigorously justify the error
bounds. We expect our approach to generalise 
easily to other parabolic fractals.}

\section{Introduction}


The Apollonian circle packing might be one of the earliest examples of fractals,
with a Hausdorff dimension strictly bigger than its topological dimension of one
and strictly less than the two dimensions of the plane of which it is a subset.
The construction dates back to 
works by  Apollonius of Perga, from 200 B.C., whose name is given to the fractal. 
The theorem attributed to Apollonius states that given~$3$ mutually tangent
circles there are exactly two other circles tangent to all of them. This starts an iterative
process: at each step one takes~$3$ mutually tangent circles from the set of existing
ones and adds two more that are tangent to all of them. The resulting limiting set is now
called the Apollonian circle packing, and we shall denote it by~$\apack$. 
Recently, its statistical properties, such as the asymptotics of the number of 
circles with curvature less than a given number, and number-theoretic properties, 
such as possible curvatures in integer packings have attracted attention of the wide community. 

In~\cite{B73} Boyd gave the first and only rigorous estimate of the
Hausdorff dimension known to date, with an accuracy of~$0.015$:
\[ 1.300197 < d_H < 1.314534. \] 
Since then, there have been a number of non-rigorous attempts to improve this
bounds. These bounds were of varying success:
\begin{itemize}
        \renewcommand{\labelitemi}{---}
	\item Various estimates have been made by counting circles in the Apollonian
        packing, typically accurate only to a few digits~\cite{MH91};
	\item Thomas and Dhar~\cite{TD94} studied the transfer operator of dynamics
        on quartets of tangent circles to obtain an estimate of
        $1.305686729(10)$. This dynamics is parabolic, and had consequently slow
        convergence; 
	\item Later, McMullen~\cite{M98} discretised the transfer operator of the
        full Apollonian circle packing. This had slow convergence, again due to
        parabolicity. He made a famous estimate of the dimension as $1.305688$,
        which turned out to be accurate only to~$5$ decimal places;
	\item More recently, Bai and Finch~\cite{BF18} discretised an induced
        transfer operator on the complex unit circle in terms of functions of
        the form $z^n \bar z^m, n,m \geq 0$. This enabled them to make a
        non-rigorous computation of the dimension to $30$ decimal places. However,
        they noted that there was no obvious way to make this computation
        rigorous. 
\end{itemize}

In the present manuscript we give an efficient and
effective method for computing the Hausdorff dimension of the limit set 
of a conformal parabolic iterated function system. Using the new approach we
give an estimate of the Hausdorff dimension of the Apollonian circle packing
accurate to~$128$ decimal places. 

\begin{theorem}
    \label{t:main}
The Hausdorff dimension of the Apollonian circle packing is 
\begin{align*}
\dim_H (\apack) = 
1.&3056867280\, 4987718464\, 5986206851\, 0408911060\, 2644149646\, \\
&8296446188\,3889969864\, 2050296986\,4545216123\,1505387132\, \\ 
&8079246688\,2421869101\,967305643 \pm 10^{-129}.
%
\end{align*}
\end{theorem}
The accuracy of this estimate could be easily improved using our algorithm subject to more computer resources and time. 
Note that all circle packings arising from three mutually tangent circles have the same Hausdorff dimension as they are conformally equivalent.

The proof of Theorem~\ref{t:main} is founded on the following ideas. 
\begin{enumerate}
    \item The classical Ruelle--Bowen formula, which relates the Hausdorff
dimension to the maximal eigenvalue of a suitable linear operator acting on a
Hardy space. We rely here on an adaptation by Mauldin and Urbanski~\cite{MU03} of Ruelle's classical
results to infinite iterated function systems.  
\item The linear operator's analyticity improving property, which allows us to
    approximate it numerically with operators of finite rank. 
We establish exponential convergence of the {\it Chebyshev--Lagrange approximation in multiple dimensions},
and compute the relevant constants explicitly. To our knowledge these results
are new and we state them in Theorem~\ref{t:Cheby} in general form.
This is one of the key results of the paper and Section~\ref{s:Chebyshev} 
is dedicated to it. 
\item 
An adaptation of min-max type estimates, which gives us a way
to obtain lower and upper bounds of the \emph{maximal eigenvalue} of the linear operator
between Hardy spaces using the \emph{eigenvector} of the finite rank
approximation. The idea to apply the min-max estimates to problems of this type appeared
in the work of the first author with M.~Pollicott~\cite{PV22}. However, the version
we use here, presented in Section \ref{s:minmax}, is a modification that
requires fewer computational resources. 
We state it in Theorem~\ref{t:PositiveOperator2} as a general result applicable
to other problems.   
\item In order to apply the min-max estimate in practice, it is necessary to be
    able to evaluate the linear operator applied to a given function with 
    arbitrary precision. This is done using the classical Euler--Maclaurin formula, building on \cite{W21}. 
    In Section~\ref{s:pointapprox} we show that error terms arising in the
    Euler--Maclaurin formula
    can be effectively and efficiently estimated in general setting of
    contracting infinite iterated function systems.
    This is one of the key results of the paper that we state in 
    Theorem~\ref{t:EMgeneric}. We give formulae that allow one to obtain estimates 
    for all constants involved in Theorem~\ref{t:PointwiseEvaluation}.
\end{enumerate} 
 In Section~\ref{s:aprioribdd} we apply the general strategy outlined above to compute the
 Hausdorff dimension of the Apollonian gasket. In particular, we evaluate
 explicitly all necessary constants involved. We give an outline of the
 algorithm used by our computer program and provide pseudocode in Section~\ref{s:algorithm}.

\section{Functional analysis setup}
In order to implement our strategy, we need to realise the Apollonian
gasket as a limit set of a contracting iterated function system. We also need to
find suitable Hardy spaces and an appropriate linear operator. 

Following Mauldin and Urbanski~\cite{MU98}, we
shall define the Apollonian gasket as the limit set of a parabolic iterated function scheme as
follows. We begin with the closed unit disk $\mathbb D = \{ z \in \mathbb C \mid
|z| \le 1 \}$ and a Moebius transformation
$$
A_\bullet \colon \mathbb D \to \mathbb D \qquad 
A_\bullet
\colon  z \to \frac{(\sqrt3-1)z +1}{-z + \sqrt 3 +1}.
$$
We shall denote by $\mathtt{e}^{\pm 2\pi /3}$ the rotations
by $\frac{-2 \pi}3$ and $\frac{2\pi}3$, i.e. 
$$
\mathtt{e}^{-2\pi /3} \colon z \mapsto e^{\frac{-2\pi}3 i} \cdot z, \qquad 
\mathtt{e}^{2\pi /3} \colon z \mapsto e^{\frac{2\pi}3 i} \cdot z.
$$
Now we may consider the iterated function system of~$3$ maps (see
Figure~\ref{fig:scheme})
\begin{equation}
    \label{eq:apIFS}
\left\{A_1 : = A_\bullet, \,
A_2: = \mathtt{e}^{2\pi/3} \circ A_\bullet, \mbox{ and } A_3: = \mathtt{e}^{-2\pi/3} \circ
A_\bullet \right\}.
\end{equation}
The limit set is then defined by 
$$
\apack = \bigcap_n \bigcup_{|\sigma| = n} A_\sigma (\mathbb D),
$$
where the union is taken over all sequences of length~$n$ of symbols $1,2,3$
and 
$$
A_\sigma := A_{\sigma_1} \circ A_{\sigma_2} \circ \ldots \circ A_{\sigma_n}.
$$

\begin{figure}[h!]
    \centerline{\includegraphics{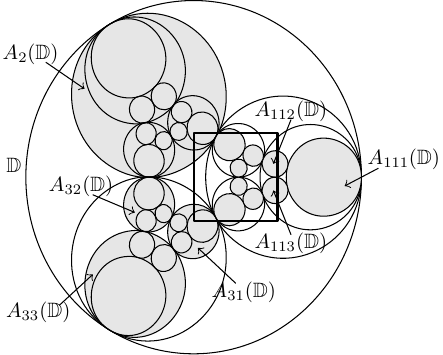} \qquad \qquad
\includegraphics{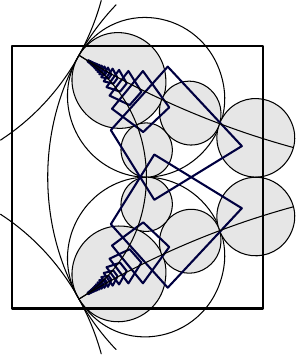}}
    \caption{Left: The unit disk and its images under iterations of $A_1$,
    $A_2$, $A_3$. The square $\square = \{ (x,y) \mid 0 \le x \le 0.5,\, -0.25
    \le y \le 0.25\}$, where the induced system is defined is shown.
    Right: The square~$\square$ and its images under the maps $f_n$ for $n = 1, \ldots,  20$.} 
    \label{fig:scheme}
\end{figure}

Although the system is very simple, it is not contracting. 
Indeed, it is easy to see that there is a neutral fixed point of $A_\bullet$ at~$z = 1$, as $A_\bullet^\prime(1) = 1$. 
It is a classic approach to use inducing to replace a finite parabolic system 
by a uniformly contracting system consisting of countably
many maps whose limit set has the same Hausdorff dimension. 

\subsection{Inducing}
Consider the square $\square:=\{ z \in \mathbb C \mid  0< \Re(z) < 0.5, \, |\Im(z)| < 0.25
\}$. We claim that the Hausdorff dimension of the circle packing is equal to the
Hausdorff dimension of the part of the packing contained in $\bigl(A_{12}(\mathbb
D) \cup A_{13}(\mathbb D) \bigr) \cap \square $. In other words, we claim 
$$
\dim_H \apack = \dim_H \left(\apack \cap \square \cap \left(A_{12}(\mathbb
D) \cup A_{13}(\mathbb D) \right) \right). 
$$
Indeed, the entire gasket can be obtained as a union of images of the part of
$\apack$ contained within $\square \cap \left(A_{12}(\mathbb D) \cup 
A_{13}(\mathbb D) \right)$ by Moebius transformations, which are bi-Lipshitz and
thus preserve the Hausdorff dimension. 

We would like to define the infinite iterated function system by 
\begin{equation}
f^\pm_n \colon \square \to \square,  \qquad
f_{n}^- = A_\bullet \circ \mathtt{e}^{-2\pi /3} \circ A_\bullet^{n}, \quad
f_{n}^+ = A_\bullet \circ \mathtt{e}^{2\pi /3} \circ A_\bullet^{n}, 
\quad n \in \mathbb N\cup\{0\}. 
\label{eq:cIFS}
\end{equation}
Since $|A_\bullet^\prime|<1$ on the image $\mathtt{e}^{\pm2\pi /3} A_\bullet(\mathbb D)$, 
the maps $f_n$ are uniformly contracting.
The limit set is $\apack^\prime = \cap_n \cup_{\pm,|\sigma|=n} f^\pm_\sigma
(\square)$, where the union is taken across all sequences of natural numbers of
length $n$ for all possible combinations of~$-$ and~$+$ indices.

We claim that the two limit sets 
have the same Hausdorff dimension. 
\begin{lemma}
    In notation and under the hypothesis introduced above, 
$$
\dim_H \apack = \dim_H  \apack^\prime. 
$$
\end{lemma}
\begin{proof}
    It is sufficient to show that for any $x \in \apack \cap (A_{12}(\mathbb D) \cup
    A_{13}(\mathbb D))$ we have $x \in \apack^\prime$.  
    Indeed, the assumption implies that there exists a sequence $\sigma =
    1,\sigma_2,\sigma_3,\ldots$ such that $x = \lim_{n\to \infty}
    A_{\underline{\sigma}_n}(y)$ for some $y \in \square$. Expanding  $A_{\underline{\sigma}_n}$ using
    formulae~\eqref{eq:apIFS} for $A_{1,2,3}$ we get 
    \begin{align*}
    A_{\underline{\sigma}_n} &= A_{\bullet} \mathtt{e}^{\pm 2\pi i/3} A_\bullet
    A_{\bullet}^{j_1} \mathtt{e}^{\pm 2\pi /3} A_\bullet 
    A_{\bullet}^{j_2} \mathtt{e}^{\pm 2\pi /3} \ldots A_{\bullet}^{j_{n-1}}
    \mathtt{e}^{\pm 2\pi /3}
    A_\bullet A_{\bullet}^{j_n}  \\ &=
    A_{\bullet} \mathtt{e}^{\pm 2\pi /3} A_{\bullet}^{j_1}  A_\bullet
    \mathtt{e}^{\pm 2\pi /3} 
    A_{\bullet}^{j_2} A_\bullet \mathtt{e}^{\pm 2\pi /3} \ldots
    A_{\bullet}^{j_{n-1}} A_\bullet \mathtt{e}^{\pm 2\pi /3}
    A_\bullet^{j_n} A_{\bullet} \\ &= 
    f_{j_1}^\pm \circ f_{j_2}^\pm \circ \ldots \circ f_{j_{n-1}}^\pm \circ
    f_{j_n+1}^\pm, \quad j_n \ge 0.
    \end{align*}
    Where the sign in $f_{j_k}^\pm$ depends on the sign in $\mathtt{e}^{\pm 2\pi /3}$ that precedes~$A_\bullet^{j_k}$.
    Thus there exist a sequence of indices $\sigma^\prime$ such that $x = \lim_{n\to \infty}
    f_{\underline{\sigma^\prime}_n}(y) \in \apack^\prime$. 
\end{proof}

It is very easy to evaluate the maps $f_n$ quickly and with arbitrary precision. 
The Moebius transformation $A_\bullet$ corresponds to the matrix 
$\bigl(\begin{smallmatrix}     \sqrt3 - 1 & 1  \\     -1 & \sqrt3 + 1
\end{smallmatrix}\bigr)$ with the Jordan normal form decomposition
$$
A_\bullet = \begin{pmatrix}     \sqrt3 - 1 & 1  \\     -1 & \sqrt3 + 1 \end{pmatrix} = 
    \begin{pmatrix}     1 & 1 \\ 1 & 0 \end{pmatrix} \cdot
    \begin{pmatrix}     \sqrt3 & -1 \\ 0 & \sqrt3 \end{pmatrix} \cdot
    \begin{pmatrix}     0 & 1 \\ 1 & -1 \end{pmatrix}.
$$
Hence 
$$
A_\bullet^n = 
    \begin{pmatrix}     1 & 1 \\ 1 & 0 \end{pmatrix} \cdot
        \begin{pmatrix}     \sqrt3 & -n  \\ 0 & \sqrt3 \end{pmatrix} \cdot
    \begin{pmatrix}     0 & 1 \\ 1 & -1 \end{pmatrix},
$$
and in particular
$$
A_\bullet^n \colon z \mapsto \frac{(\sqrt3-n)z+n}{-nz+n + \sqrt3}.
$$

However, in order to define Hardy spaces for the transfer operator to act on, we need to work with contracting
real maps, meaning we need to replace the complex line~$\mathbb C$ with~$\mathbb
R^2$. 
Below we describe one possible construction, that in particular takes care of \emph{complex}
values of~$n$, which will be essential later. 

Following~\cite{H89} the complex conjugate of~$f$ may be defined by $\bar f(z)
:= \inv f \inv $, where $\inv \kern-4pt : \! z \mapsto \overline z$ is the fundamental involution. 
Then $\bar f$ is analytic on the domain $\{ I( z) \mid f \mbox{ is analytic at
} z \}$.
We now can convert our maps $f_n^\pm: \C \to \C$ into real-analytic maps of two real variables as 
\begin{equation}
    \label{eq:Rmaps}
{\widehat G}_n^{\pm}\colon
(x,y) \mapsto \left(\tfrac{1}{2}(f^\pm_n(x+iy) + \overline{ f^\pm_n}(x-iy)),
\tfrac{1}{2i}(f^\pm_n(x+iy) - \overline{ f^\pm_n}(x-iy))\right).
\end{equation}
Similarly, the derivatives of 
the maps $f^\pm_n: \C \to \C$ become real-analytic functions $\R^2 \to \R$ given by
\begin{equation}
	\label{eq:RmapsJac}
    {\widehat J}^\pm \colon	(x, y)  \mapsto \sqrt{(f_n^\pm)'(x+iy) 
    (\overline{f_n^\pm})'(x-iy)}, \quad x, y \in \mathbb R. 
\end{equation}
Since $f_n^\pm$ are Moebius maps, $(f_n^\pm)^\prime$ don't vanish and therefore the positive
branches can be consistently chosen across complex domains. 

Conjugating with the mapping $S \colon [-1,1]^2 \to \square$, $S \colon z
\mapsto \frac{z+1}4$ we get
an infinite contracting iterated function system $G_n^\pm \colon [-1,1]^2 \to [-1,1]^2$
given by $G_n^{\pm} := S^{-1} \widehat G_n^{\pm} S$ with Jacobian $J_n^\pm := \widehat J_n^{\pm} S$. Since~$S$  is
bi-Lipshitz, the limit set of $\{G_n^{\pm}\}_{n \in \mathbb N}$ has the same
Hausdorff dimension as~$\apack$. Therefore we set ourselves to the problem
of computing the Hausdorff dimension of the limit set of~$\{G_n^{\pm}\}_{n \in
\mathbb N}$. 

Going forward we identify $G_n^\pm$ and $J_n^\pm$ with their natural
extensions to~$\mathbb C^2$, and, furthermore, take natural extensions to
complex values of~$n$.

\subsection{Transfer operators and Ruelle--Bowen formula}
Recall that to any $r > 0$ we may associate the Bernstein ellipse with 
semi-axes $\cosh r$ and $\sinh r$ defined by
$$
	\mathcal{E}_r := \left\{ z \in \C \, \Bigl| \, \frac{(\Re z)^2}{\cosh^2 r} + \frac{ (\Im
  z) ^2}{\sinh^2 r} \leq 1 \right\}. 
$$
These ellipses contain the real interval $[-1,1]$, and play an important role in
bounding the errors of the Chebyshev approximation that we will rely upon in our numerical estimates. 
As our dynamics act over two real dimensions, we shall need a $\mathbb C^2$ analogue. Observe that we can rewrite the ellipse as 
$\mathcal E_r = \cos(\R + i[-r,r]) \supseteq [-1,1]$. This generalises to polydisks\footnote{Here and below by
$\cos(z_1,z_2, \ldots, z_d)$ we understand
$(\cos(z_1),\cos(z_2),\ldots,\cos(z_d))$; and for a set~$A$ we use shorthand
notation $\cos(A) = \{ \cos(x) \mid x \in A \mid \}$.}
\begin{equation}
  \label{eq:el2}
  \mathcal E_r^{2,2} := \cos \left(\R^2 + i \cdot \{ x \in \mathbb R^2 \mid \|x\| < r \}  \right) \subset \C^2.
\end{equation}
It is known~\cite{R69} that the set of bounded analytic functions on $\mathcal
E_r^{2,2}$ is a Banach space when considered with the supremum norm. 
Such a space of complex functions is commonly referred to as a Hardy space.

We will consider a closed subspace consisting of functions that map the square
$[-1,1]^2 \subset \mathbb R^2$ to the real line. More precisely, we want to
consider\footnote{Here and below $\overline{\mathcal E_r^{2,2} }$ stands for the
closure of the ellipsoid $\mathcal E_r^{2,2}$.}
\begin{equation}
\mathcal H_r^{2,2} = \{ \varphi \colon \mathcal E_r^{2,2} \to \mathbb C \mid 
\varphi \in C^\omega(\mathcal E_r^{2,2}) \cap C^0 (\overline{\mathcal
E_r^{2,2} }), \ \varphi([-1,1]^2) \subseteq \R \}.
\end{equation}
Since the maps $G_n^{\pm}$ are contracting, for any given $n\in\N$ there exist $ r < R$ such that
$G_n^{\pm} (\mathcal E_R^{2,2}) \subset \mathcal E_r^{2,2}$. Moreover, we show
in Section~\ref{s:aprioribdd} that the choice $R = R_A = 1.4$ and $r = r_A = 0.9$ works for all $n\in \mathbb N$. 

One key stipulation on the Jacobian is that, in order to use the discrete Taylor series approximation to estimate the Euler--Maclaurin
integral (see Proposition~\ref{p:EMIntegralError}), we will need control over its analytic extension at $n = \infty$: 

\begin{proposition}
  \label{p:AnalyticAtInfinity}
	The functions $n \mapsto n^{2s} J^{\pm}_n(z_1,z_2)^{2s}, n \mapsto
    G_n^\pm(z_1,z_2)$ are analytic and bounded for complex $n$ as $n \to \infty$, uniformly for
    $(z_1,z_2) \in \overline{\mathcal E_r^{2,2}}$, provided~$r$ is chosen so
    that $|z_1 +  iz_2 - 3| > \varepsilon  > 0$ for all $(z_1,z_2) \in
    \overline{\mathcal E_r^{2,2}}$ for some $\varepsilon>0$. 
\end{proposition}

\begin{proof}
    Going back to the definition, $\widehat G_n^\pm$ correspond to the Moebius
    mapping with the matrix 
    \begin{equation}
    \begin{pmatrix}     \sqrt3 -1 & 1 \\ -1 & \sqrt 3+1  
    \end{pmatrix} \cdot
    \begin{pmatrix}    (\sqrt3-n) e^{\pm i2\pi/3} & n e^{\pm i2\pi/3} \\ -n  & n +
        \sqrt3 \end{pmatrix} \\
 \label{eq:fnmat}
    \end{equation}
    To obtain matrices corresponding to $G_n^{\pm}$ we need to conjugate with $S(z) = (z+1)/4$, giving
    \begin{equation}
       S f_n^\pm S^{-1} \ \Leftrightarrow \  \begin{pmatrix} 1 & -\frac14 \\ 0 & \tfrac14    \end{pmatrix} \cdot
    \begin{pmatrix}     \sqrt3 -1 & 1 \\ -1 & \sqrt 3+1     \end{pmatrix} \cdot
    \begin{pmatrix}    (\sqrt3-n) e^{\pm i2\pi/3} & n e^{\pm i2\pi/3} \\ -n  & n +  \sqrt3 \end{pmatrix} \cdot
    \begin{pmatrix} 1 & 1 \\ 0 & 4    \end{pmatrix}. 
\label{eq:fnsmat}
    \end{equation}
    Using the formula~\eqref{eq:Rmaps} we conclude that $G_n^{\pm}(z_1,z_2)$ is analytic
    on $\overline{\mathcal E_R^{2,2}}$ provided that for all  $(z_1,z_2) \in
    \overline{\mathcal E_r^{2,2}}$ the compositions $S {f_n^\pm} S^{-1}$ and $S
    \overline{f_n^\pm} S^{-1}$ are
    analytic in an $\varepsilon$-neighbourhoods of $z_1 + i z_2$ and $z_1 -
    iz_2$, respectively. Since $\mathcal E_R^{2,2}$ is symmetric in
    $z_2$, it is enough to consider $z_1+iz_2 \in \mathcal E_r^{2,2}$.

    Now, a Moebius transformation given by matrix $(\begin{smallmatrix} 
    	a & b \\ c & d
    \end{smallmatrix})$ is analytic away from $-\frac{c}d$, so we can deduce
    that $(S f_n^\pm S^{-1})(z_1+ iz_2)$ are analytic when $z_1 + iz_2$ are
    bounded uniformly away from 
  \begin{align*}
   \frac{3 \left(2 + \sqrt 3 \pm i \right)  n + 9 + 8 \sqrt
  3 \pm i }{\left( 2 + \sqrt 3 \pm i \right)  n-  1 \mp \sqrt3 i}.
  \end{align*}
This will occur for all complex $|n|$ large enough if all combinations $z_1 + i z_2 \in \C$ are uniformly bounded away from from  
\[ \lim_{n\to\infty} \frac{3 \left(2 + \sqrt 3 \pm i \right)  n + 9 + 8 \sqrt 3
\pm i }{\left( 2 + \sqrt 3 \pm i \right)  n-  1 \mp \sqrt3 i }= 3. \] 
   At the same time the formula~\eqref{eq:RmapsJac} for Jacobians gives that
    \begin{align*}
        n^{2s} J_n^\pm(z_1,z_2)^{s} = \left(n \sqrt{(S f_{n}^\pm S^{-1})'(z_1+
        iz_2)} \cdot n
        \sqrt{(IS f_{n}^\pm S^{-1}I)'(z_1 - iz_2)}\right)^s.
 \end{align*}
This will be bounded and analytic in $n$ in a neighbourhood of infinity provided
that 
\begin{enumerate}
    \item[(a)] $n \mapsto f_n^\pm( z_1 + i z_2 )$ are analytic at infinity; and 
    \item[(b)]  $n \sqrt{(S f_{n}^\pm S^{-1})'(z_1 + iz_2)}$ is bounded both from above and away from zero as $n\to\infty$. 
\end{enumerate}
It follows from the formula for $f_n^\pm$ that (a) holds; since $A_\bullet^n$ is
bounded and therefore analytic at infinity. To confirm (b) one can
find from~\eqref{eq:fnsmat} by straightforward computation that
\[ \lim_{n\to\infty} n \sqrt{(S f_{n}^\pm S^{-1})'(z_1 + i z_2)} = \frac{3
e^{\pm i\pi/3}}{(2 + \sqrt3 \pm i)(z_1 + i z_2-3)}. \]

It follows by continuity of the maps and compactness of the closed ellipse $\overline{\mathcal E_r^{2,2}}$
that the convergence is uniform in $z_1$ and $z_2$. 
\end{proof}

We next want to define a family of Apollonian transfer operators 
$\mathcal A_s \colon \mathcal H_r^{2,2} \to \mathcal H_R^{2,2}$ by 
\begin{equation}
\label{eq:ApTrOp}
[\mathcal A_s \varphi](z_1,z_2) = \sum_n \left(J_n^+(z_1,z_2)^s \cdot
\varphi (G_n^+(z_1,z_2)) + J_n^-(z_1,z_2)^s \cdot \varphi (G_n^-(z_1,z_2))
\right).
\end{equation}

The following result, connecting the Hausdorff dimension to the spectral radius
of the transfer operator is fundamental for the subsequent considerations. 
\begin{proposition}\label{p:SpecRadIsDimension}
    For $s \in (1, 2)$ the spectral radius $\rho(s):=\rho_{\mathcal H}(\mathcal A_s)$
    of $\mathcal A_s$ acting on $\Hd[2,2]{R_A}$ is a
continuous monotone decreasing function and is equal to the maximal eigenvalue.
Furthermore, $\rho(s) = 1$ when $s = \dim_H \apack$.
\end{proposition}
\begin{proof}
    Using results from the book by Mauldin and Urbanski~\cite{MU03}, we deduce
    that $\rho_{C^\alpha}(\mathcal A_s)$, the
    spectral radius of $\mathcal A_s$ on the space of \emph{H\"older continuous
    functions} on $C^\alpha([-1,1]^2)$, is a monotone decreasing function and is
    equal to the simple (maximal) eigenvalue. Furthermore, 
    $\rho(\mathcal A_s) = 1$ when $s = \dim_H \apack$. 

    To be more specific,~\cite[Theorem~2.3.3]{MU03} allows us to identify the topological
pressure $P(s)$ with $\log \rho_{C^\alpha}(\mathcal A_s)$. From \cite[Proposition~4.2.8]{MU03} we
have that $\rho_{C^\alpha}(\mathcal A_s)$ is a decreasing function of~$s$.
Then~\cite[Theorem~4.2.13]{MU03} gives that there is a unique~$s$ such that 
$\rho(\mathcal A_s) = 1$, and this is the Hausdorff dimension of the limit set~$\apack$. 
To conclude, by~\cite[Theorem 2.4.6]{MU03} $\mathcal A_s$ has a spectral gap
on~$C^\alpha([-1,1]^2)$.  

    It remains to show that the spectral radius of~$\mathcal A_s$ on the Hardy
    space $\Hd[2,2]{R_A}$ is equal to the spectral radius of $\mathcal A_s$ on
    $C^\alpha([-1,1]^2)$, namely $\rho_{C^\alpha}(\mathcal A_s) = \rho_{\mathcal
    H}(\mathcal A_s)$

    By straightforward computation  we establish in Proposition~\ref{p:APrioriConstants}\ref{res:CLJacobianBound}
   that for any $s > 1/2$  
	$$
    \| \mathcal A_s \|_{\Hd[2,2]{R_A}} \leq \sum_{\pm,n} \|J_n^\pm
    \|_{\Hd[2,\infty]{R}}^s < \infty. 
    $$
    Applying~\cite[Proposition~5.4]{BJ07} we conclude that $\mathcal A_s$ is a
    compact operator on $\Hd[2,2]{R_A}$. 
    
    We now want to apply~\cite[Lemma A.1]{BT08} to deduce the equality between
    the peripheral spectra on $\Hd[2,2]{R_A}$ and  $C^\alpha([-1,1]^2)$. Indeed, we
    may choose $B_0 = B_1 = \Hd[2,2]{R_A}$ and $B_2 = C^\alpha[-1,1]^2$. Then
    $B_0$ is dense in $B_2$; and both $B_0$, $B_2$ are invariant with respect to
    $\mathcal A_s$, so the hypothesis of the Lemma holds.  
\end{proof}

Therefore our goal is to compute the value of the parameter~$s$ such
that~$\mathcal A_s$ has maximal eigenvalue~$1$ when acting on $\mathcal
H_{R_A}^{2,2}$. Following the strategy outlined
above, our first step is to construct a finite rank approximation. In the next
section we shall show that for analyticity improving weighted composition operators
the Chebyshev--Lagrange approximation converges
exponentially. The transfer operator~\eqref{eq:tropgen} will be a special case. 

\section{Chebyshev approximation in multiple dimensions}\label{s:Chebyshev}

In this section we generalise the approach of Bandtlow and Slipantschuk~\cite{BS20} to multiple dimensions. Our
approximation of the Apollonian circle packing's transfer operator will be a
special case corresponding to the choices dimension~$d=2$ and spatial~$\ell^p$ metric~$p=2$. 

We will first define the Hardy spaces in which this approximation occurs; we
will then define the projection operator and give our main theorem on transfer
operators, Theorem~\ref{t:Cheby}. Finally we will prove the main approximation
result, Proposition~\ref{p:ProjectionNorms}. 

\subsection{Hardy spaces}

It is easy to generalise the construction of the Bernstein ellipses in $\mathbb C^2$
given in~\eqref{eq:el2} to $\mathbb C^d$. In the future, it will be more convenient
to use the~$\ell_p$ norm on~$\mathbb R^d$ for $p \in [1,\infty]$. So let us define
$$
\mathrm{Ball}_r^{d,p} = \left\{ x \in \mathbb R^d \mid \|x\|_p < r \right\}.
$$
By analogy with~\eqref{eq:el2} we introduce
the $d$-{\it dimensional Bernstein ellipses} 
\begin{equation}
  \label{eq:elb}
\Bd{r} := \cos (\R^d + i \mathrm{Ball}_{r}^{d,p}) \subseteq \Bd[d,\infty]{r} =
(\mathcal{E}_r)^d \subset \C^d.
\end{equation}
	Note that when $d=1$ or $d=2$ this agrees with the $\mathcal E_r$ and
    $\mathcal E_r^{2,2}$ as defined before.

 Next we want to consider a Hardy space of holomorphic functions on
 the interior of $\Bd{r}$ continuous onto the boundary. 
Similarly to the case $d=2$ we introduce a
closed subspace consisting of functions that map the real cube
$[-1,1]^d$ to the real line. More precisely, we set
\begin{equation}
  \label{eq:hardy_rdp}
\mathcal H_r^{d,p} = \{ \varphi \colon \mathcal E_r^{d,p} \to \mathbb C \mid 
\varphi \in C^\omega(\mathcal E_r^{d,p}) \cap C^0 (\overline{\mathcal
E_r^{d,p}}), \ \varphi([-1,1]^d) \subseteq \R \}.
\end{equation}
with the supremum norm 
 \begin{equation}
     \|\varphi\|_{\mathcal H_r^{d,p}} = \sup_{\z \in \Bd{r}} |\varphi(\z) |. 
 \label{eq:hdNorm}
\end{equation}

The transfer operator for the Apollonian packing defined in~\eqref{eq:ApTrOp} is
a weighted composition operator acting
on~$\mathcal H_r^{2,2}$ spaces. 

More generally, let us assume that the sizes of the ellipses $0<r < R$, as well
as~$p$ and~$d$ are fixed.
Given countable families of weight functions $w_j \in \mathcal H_R^{d,p}$ and
analytic maps $v_j
\colon \mathcal E_R^{d,p} \to \mathcal E_r^{d,p}$, we define the transfer operator 
in a general form 
\begin{equation}
  \label{eq:tropgen}
\mathcal L \colon \mathcal H_r^{d,p} \to \mathcal H_R^{d,p}, 
\qquad \mathcal L \varphi = \sum_j w_j \cdot (\varphi \circ v_j). 
\end{equation}
The following proposition shows that in order to guarantee that the operator is
bounded it is sufficient to impose an assumption of
summability of weights. 
\begin{proposition}\label{p:TransferBounded}
Assume that $W < \infty$ is chosen such that
	$\sum_{\iota \in I} \bigl\| w_\iota \bigr\|_{\Hd{R}} \leq W $. Then
  for $0< r < R$ as above, the transfer operator~\eqref{eq:tropgen} is  a bounded operator $\Hd{r}\to \Hd{R}$ with
	\[ \| \L \|_{\Hd{r}\to \Hd{R}} \leq W. \]
\end{proposition}
\begin{proof}
	From the definitions of $\L$ and the norm on $\Hd{R}$ we have
	\[ \| \L \psi \|_{\Hd{R}} = \sup_{\z \in \Bd{R}} \Bigl|  \sum_{\iota \in I}
  w_\iota(z)\ \psi(v_\iota(z)) \Bigr|. \]
	Since $v_\iota$ maps $\Bd{R}$ into $\Bd{r}$ we can bound
  $$
	 \| \L \psi \|_{\Hd{R}} \leq \sup_{\z \in \Bd{R}} 
    \sum_{\iota \in I} \left| w_\iota(z)\right|\ \| \psi\|_{\Hd{r}}  
	\leq \sum_{\iota \in I} \left\| w_\iota \right\|_{\Hd{R}} \|\psi\|_{\Hd{r}}
  \leq W \cdot  \|\psi\|_{\Hd{r}}, 
  $$
	as required.
\end{proof}

In the next section we explain how to approximate the transfer operator by
finite rank operators.
\subsection{Chebyshev approximation in multiple dimensions}

The Chebyshev polynomials 
\begin{equation}
    \label{eq:TkDef1}
  	T_k(z) = \cos (k \arccos z), \ k \in \N 
\end{equation}
	are a family of univariate polynomials. They are orthogonal with respect to
  the inner product 
	$$
  \langle \varphi, \psi \rangle = \int_{-1}^1 \frac{\overline{\varphi(x)}
  \psi(x)}{\sqrt{1-x^2}}  \d x  .
  $$
	They are very well-adapted to smooth function approximation on intervals,
  which can be partly attributed to their close connection to Fourier series.
  One consequence of this connection is that the polynomials $T_k$,
  $0 \le k < K$ are mutually orthogonal with respect to the counting measure on the
  {\it Chebyshev nodes of the first kind} of order $K$: 
  \begin{equation} x_{k,K} = 
      \cos \tfrac{\pi (2k+1)}{2K}, \ k = 0,1,\ldots, K-1.
      \label{eq:ChebyshevNodes} 
  \end{equation} 
  and form a basis in the space of polynomials of degree~$K-1$. 
	The transformation between values of a function at $x_{k,K}$ and its
  coefficients in the basis of Chebyshev polynomials can be computed in $O(K\log K)$
  time using the Fast Fourier Transform~\cite{T12}. 
	
	To approximate smooth functions on $d$-dimensional hypercubes we will use a
  tensor product of Chebyshev bases, with the new basis functions being  
  \begin{equation}
    T_{\k}(\z) = \prod_{j=1}^d T_{k_j}(z_j),\ \k \in \N^d.
  \label{eq:TkDef2}
  \end{equation}
	The collection of tensor product functions 
  $\{T_{\k}: \|\k\|_{\ell_\infty} < K \}$ is orthogonal with respect to
  counting measure on the (tensor product) Chebyshev nodes 
  \[ \left\{ \vec x_{\k,K} = (x_{k_1,K}, x_{k_2,K}, \ldots, x_{k_d,K}) \mid \|\k\|_{\ell_\infty} < K \right\}. \]
  By analogy with the one-dimensional case, we introduce the basis of the space
  of polynomials $\langle
  T_{\k} \mid \|\k\|_{\ell_\infty} < K
  \rangle$ consisting of the Chebyshev--Lagrange polynomials by
\begin{equation}
\left\{  \ell_{\vec k,K}(\vec y) = \prod_{j=1}^d \ell_{k_j,K}(y_j),
\|\k\|_{\ell_\infty} < K \,\Bigl|\, \ell_{k_j,K}(y) := \prod_{i=1,\ldots,K, i\neq k_j}
\tfrac{y-x_{i,K}}{x_{k,K}-x_{i,K}},\, \right\}. 
\label{eq:LagrangePoly}
\end{equation}
Note that here we use an~$\ell^\infty$ norm independent of the choice of~$p$ parameter.

Let the projection operator $\P_K \colon \mathcal H_\rho^{p,d} \to \langle
  T_{\k} \mid \|\k\|_{\ell_\infty} < K
  \rangle$ be interpolation on these nodes, that is 
	\begin{equation}
        \label{eq:projOp}
		(\P_K \phi)(\vec y) = \sum_{\vec k \in \{1,2,\ldots,K\}^d} \phi(\vec x_{\vec k,K}) \ell_{\vec k,K}(y)
	\end{equation}

The following results, extending those of \cite{BS20},
  show that the transfer operator $\L$ is nicely bounded in these Hardy spaces,
  and is readily approximable using interpolation at Chebyshev nodes. 

  In what follows,~$\idop$ stands for the natural embedding between the spaces
  involved. 

	\begin{proposition}\label{p:ProjectionNorms}
  For any  $R>\rho\geq 0$ and $K \in \N$ the projection error is bounded by
	  	\[ \| \idop - \P_K \|_{\Hd{R} \to \Hd{\rho}} \leq E_{d,K}(R-\rho).\]
        where  $E_{d,K}$ is the error function given by\footnote{The estimate
        for $d \ge 4$ can be improved by computing the incomplete Gamma function
        exactly --- see~\eqref{eq:sumub0}.} 
 \begin{equation} 
     \label{eq:erFeps}
     E_{d,K}(x) = 
     \begin{cases}
 	         8 e^{-(K-1)x}\cdot\frac{1}{x}, & \mbox{ if } d = 1; \\
         16 e^{-(K-1)x}\cdot\frac{1+Kx}{x^{2}}, & \mbox{ if } d = 2; \\
         48 e^{-(K-1)x}\cdot\frac{2+ 2 Kx +  K^2 x^2}{x^{3}},  & \mbox{ if } d =3; \\
         2^d e^{1-(K-1)x} \cdot \frac{d^2  K^{d-1} }{x}      
     & \mbox{ if }          d \ge 4.
     \end{cases}
 \end{equation}

\end{proposition}
In other words, the more the maps $v_\iota$ contract, 
the bigger the difference in sizes between the ellipses $\mathcal
E_R^{d,p}$ and $\mathcal E_r^{d,p}$ that the contraction maps one into another
is; and consequently the smaller the projection error is.
Proposition~\ref{p:ProjectionNorms} facilitates approximation of the transfer
operator by finite-rank operators. The proof of
Proposition~\ref{p:ProjectionNorms} is rather technical and we postpone it until
\S\ref{s:propProjNormProof}. The most important consequence for us is the next 
theorem, which is its immediate consequence. 

\begin{theorem}\label{t:Cheby}
Given $0 < r < R$, $p, d \in \N$ and countable families of
weight functions $w_j \in \mathcal H_R^{d,p}$ and analytic maps $v_j
\colon \mathcal E_R^{d,p} \to \mathcal E_r^{d,p}$ satisfying 	$
\sum \| w_j \|_{\Hd{R}} \leq W $, consider the transfer operator
$$
\mathcal L \varphi = \sum_j w_j \cdot (\varphi \circ v_j). 
\eqno{\eqref{eq:tropgen}}
$$
		Then for all $K \in \N$ and $\rho \in [0,R)$ we have 
		\[ \| \P_K \L - \L \|_{\Hd{r}\to\Hd{\rho}} \leq W \cdot
        E_{d,K}(R-\rho),\]
        where $ E_{d,K}$ is the error function given
        by~\eqref{eq:erFeps}.
  \end{theorem}
	\begin{proof}
	Using Propositions~\ref{p:TransferBounded} and~\ref{p:ProjectionNorms}, we have
		\[ \| \P_K \L - \L \|_{\Hd{r} \to \Hd{\rho}} \leq \| \P_K - \idop \|_{\Hd{R}
        \to \Hd{\rho}} \| \L \|_{\Hd{r} \to \Hd{R}} \leq W \cdot E_{d,K}(R-\rho), \]
		as required.
	\end{proof}

\begin{corollary}
    \label{cor:trCompact}
Under the hypothesis of Theorem~\ref{t:Cheby} the
operator~$\mathcal L$ is a compact endomorphism on~$\Hd{r}$.
\end{corollary}
\begin{proof}To show that $\L$ is compact, we need only show that the inclusion $\Hd{R} \to
\Hd{r}$ is compact. Proposition \ref{p:ProjectionNorms} gives that this
inclusion is approximable in operator norm by a sequence of finite-rank and
hence compact operators $\{\P_K \}_{K \in \N}$, and is therefore compact. 
\end{proof}

Since the (non-trivial) eigenfunctions of $\P_K \L$ lie in 
$\im \P_K = \spann\{T_\k \mid \| \k \|_\infty \leq K \}$,
the spectral data of $\P_K\L$ can be obtained from the truly finite-dimensional
restriction $\P_K \L |_{\im \P_K}$, which is given by a matrix in the basis of
Chebyshev--Lagrange polynomials~\eqref{eq:LagrangePoly}.
We shall use Theorem~\ref{t:Cheby} to obtain a good approximation to the spectral radius of~$\L$. 

\begin{remark}
Extending the definition of C-expansion from~\cite{W19}, we can expect
that, at least for a fixed $j$, $R$ and $r$ satisfying hypothesis of
Theorem~\ref{t:Cheby} exist provided $v_j$ and $w_j$ 
  are real-analytic with $v_j$ obeying
  the following C-contraction condition: 
		\[ \sup_{\btheta \in [0,\pi]^d} \left\|(\arccos (  v_j ( \cos
        (\btheta)))^\prime_{\btheta} \right\|_{\ell^p} < 1 \]
        where $\cos$ is applied component-wise, and $\btheta \in [0,\pi]^d$ can
        be thought of as the arc-cos of the position variable $\x \in [-1,1]^d$.
        (The $\cos$ connection arises from the connection between Fourier and
        Chebyshev functions: $T_k(\cos \theta_j) = \cos k\theta_j$.) 
	\end{remark}

\begin{remark}
    In practice, the ambient dimension $d$ will be fixed by the problem and
generally low: the Apollonian circle packing IFS has $d=2$. We shall show later
in \S\ref{s:aprioribdd} that for the Apollonian packing we can choose $R_{A} =
1.4$ and $r_{A} = 0.9$. It turns out (see Table~\ref{tab:Scalings}) that in order to 
guarantee that the dimension is computed with the error~$\varepsilon$ it is
sufficient to consider polynomials of degree $K \ge -\log
\varepsilon/R_{A}$. 
\end{remark}

To complete this section it only remains to prove Proposition~\ref{p:ProjectionNorms}. 

\subsection{Proof of Proposition~\ref{p:ProjectionNorms}  }
\label{s:propProjNormProof}

Here we show that the projection operator~\eqref{eq:projOp} is close to the identity
and establish the formula~\eqref{eq:erFeps} for the bound~$E_{d,K}$ on the error.

We need two easy technical statements that will simplify the exposition of the
argument.
\begin{lemma}\label{l:ModNEllQNorm}
	Let ${\modck_K} (k_j) := \left|\left( (k_j + K) \mod 2K\right) - K\right|$, and define ${\modck_K} (\k)$
    component-wise. Then, for all $\k\in \N^d$ and $q \in [1,\infty]$,
    $\|{\modck_K} (\k)\|_{q} \leq \|\k\|_{q}$. 
\end{lemma}
\begin{proof}
	It suffices to show that $|{\modck_K} (k_j)| \leq |k_j|$. Indeed, we have that
    either $k_j \leq K$,
    in which case ${\modck_K} (k_j) = k_j$, or $k_j > K$, in which case $|{\modck_K}
    k_j| \leq K < |k_j|$. 
\end{proof}

\begin{lemma}
    \label{lem:h2}
	For all $\k \in \N^d$, $\P_K T_{\k} = c_{\k,K} T_{{\modck_K} (\k)}$, where
    $c_{\k,K} \in \{-1,1,0\}$. In particular, $\P_K T_{\k} = T_{\k}$ when
    $\|\k\|_{\ell^\infty} < K$ and $c_{\k,K} = 0$ when there exists $j$ such
    that $k_j \mod 2K = K$. 
\end{lemma}
\begin{proof}
  The statement is established in~\cite{T12}~for $d=1$. The case for higher
  dimensions follows by multiplicativity.
\end{proof}
\begin{lemma}
     \label{eq:|T|rho}
     For all $\rho > 0$ and for all $\k \in \mathbb R^d$ and 
		$q^{-1} + p^{-1} = 1$ we have
     $$
        \| T_{\k} \|_{\Hd{\rho}} \leq e^{\rho \| \k \|_q}.
        $$
\end{lemma}
\begin{proof}
     Combining the definitions of the ellipses~$\mathcal
        E_\rho^{d,p}$~\eqref{eq:elb}, the norm on $\Hd{\rho}$~\eqref{eq:hdNorm},
        and the Chebyshev
        polynomials~$T_k$~\eqref{eq:TkDef2}, with the Holder inequality 
        $\sum k_j \vartheta_j \le \|k\|_q \cdot \|\vartheta\|_p$, 
        we obtain an upper bound
		\begin{equation}
            \| T_\k \|_{\Hd{\rho}} = \sup_{x, \vartheta \in \R^d, \|\vartheta\|_p \leq 1} 
        \prod_{j=1}^d |\cos (k_j (x_j + i \rho \vartheta_j))|
		\leq \sup_{\|\vartheta\|_p \leq 1} \prod_{j=1}^d e^{k_j \rho \vartheta_j}
		= e^{\rho \|\k\|_q}. \label{eq:ChebyshevPolynomialBounds}
		\end{equation}
          \end{proof}
    Together with Lemma~\ref{l:ModNEllQNorm}, this yields 
    \begin{corollary}
        \label{cor:|T|rho}
        $$
        \| T_{\modck_K \k} \|_{\Hd{\rho}} \leq e^{\rho \| \k \|_q}.
        $$
    \end{corollary}

    \begin{corollary}
        \label{cor:|phi|H}
       For all $\rho > 0$ and for all $\varphi \in \Hd[2,2]{\rho}$
       \[ \| \varphi \|_{\Hd[2,2]{\rho}} \leq \sum_{\k \in \N^2} e^{\rho \|
    \k\|_{\ell^2}} |\check \varphi_\k| \]
	whenever the right-hand side is finite 
    \end{corollary}
    \begin{proof}
    We can formally expand $\varphi$ in the Chebyshev series $\varphi = \sum\limits_{\k \in \N^2} \check\varphi_\k T_\k$. We then have
	\[ \| \varphi \|_{\Hd[2,2]{\rho}} \leq \sum_{\k \in \N^2} | \check\varphi_\k
    | \| T_\k\|_{\Hd[2,2]{\rho}} \leq \sum_{\k \in \N^2} | \check\varphi_\k |
    e^{\rho \|\k\|_{\ell^2}}. \]
	where in the last inequality we used \eqref{eq:ChebyshevPolynomialBounds} on
    each term. The convergence of this sum justifies the use of the Chebyshev
    series.	 
\end{proof}

	\begin{proof}[Proof of Proposition \ref{p:ProjectionNorms}]
        By construction of $\Hd{R}$ a function $\varphi \in \Hd{R}$ is
        continuous on $[-1,1]^d$, and therefore its Chebyshev series 
        converges. More precisely, we define the Chebyshev coefficients by 
  		\begin{equation} 
            \Check \varphi_\k := \frac{t_\k}{\pi^d} \int_{[-1,1]^d}\frac{
            T_\k(\x) \varphi(\x)}{\prod_{j=1}^d\sqrt{1-x_j^2}} \d \x, \quad
        \mbox{ where } t_\k := \prod_{j=1}^d (2 - \delta^{k_j}_{ 0}). 
            \label{eq:ChebyCoeffs} 
        \end{equation}
        and the Chebyshev series expansion
		\begin{equation} 
            \varphi = \sum_{\k \in \N^d} \Check \varphi_\k T_\k. 
            \label{eq:ChebySeries} 
        \end{equation}
Applying Lemma~\ref{lem:h2} we get  
		$$
        (\idop - \P_K) \varphi = \sum_{\k \in \N^d} \Check \varphi_\k (T_\k - \P_K T_\k) = 
        \sum_{\|\k\|_\infty \geq K } \Check \varphi_\k (T_\k - c_{\k,K} T_{\modck_K  \k} )
        $$
		wherever \eqref{eq:ChebySeries} converges. Assuming the Chebyshev series
        of~$\varphi$ converges on $\Bd{\rho}$, we may write
		\begin{align}
            \| (\idop - \P_K) \varphi \|_{\Hd{\rho}} &\leq
            \sum_{\|\k\|_\infty \geq K} |\check \varphi_\k| \cdot (\| T_\k
            \|_{\Hd{\rho}} + \| T_{\modck_K \k} \|_{\Hd{\rho}}).
        \end{align} 
        Applying Corollary~\ref{cor:|T|rho} we get
		\begin{equation} 
            \| (\idop - \P_K) \varphi \|_{\Hd{\rho}} \leq
            2 \sum_{\|\k\|_\infty \geq K} |\check \varphi_\k| \cdot e^{\rho \| \k \|_q}.
            \label{eq:ProjErrorAsCoefs} 
        \end{equation} 
		We see that if the Chebyshev coefficients $\Check \varphi_\k$ decay faster
    than~$e^{-\rho \|\k\|_q}$, then $ \| (I - \P_K) \varphi \|_{\Hd{\rho}}$
    converges to zero as $K \to \infty$. 
		
    It only remains to estimate the coefficients. Recall that $T_{k_j}(x_j) = \cos(k_j
    \arccos x_j)$, and let us denote  $\varphi(\cos\btheta) := \varphi(\cos
    \theta_1,\cos\theta_2,\ldots,\cos\theta_d)$.
    We can rewrite Definition \eqref{eq:ChebyCoeffs} as an integral in an angle variable
    $\theta_j = \arccos x_j$: 
		\[  \check \varphi_\k = \frac{t_\k}{\pi^d} \int_{[0,\pi]^d}
    \left(\prod_{j=1}^d \cos k_j \theta_j\right)\,
    \varphi(\cos\btheta)\, \d\btheta. \] 

    Since the cosine is an even analytic $2\pi$-periodic function on the real line,
     we can write  
		\[  \check \varphi_\k = \frac{t_\k}{(2\pi)^d} \int_{[0,2\pi]^d}
    \left(\prod_{j=1}^d \cos k_j \theta_j\right)\,
    \varphi(\cos\btheta)\, \d\btheta. \] 
		Rewriting the cosine terms using $\cos (k_j \theta_j) = \frac12(e^{- i k_j
    \theta_j} + e^{i k_j \theta_j} )$, we get 
		\[  \check \varphi_\k = 2^{-d} \sum_{\sigma \in \{-1,1\}^d}
    \frac{t_\k}{(2\pi)^d} \int_{[0,2\pi]^d} e^{i \sum_j \sigma_j k_j \theta_j}\,
    \varphi(\cos\btheta)\, \d\btheta. \] 
		
    Note that $\varphi \in \Hd{R}$ is analytic on $\mathcal E_{R}^{d.p}$ and
    the composition $\varphi(\cos(\cdot))$ is analytic and $2\pi$-periodic on $\R^d +
    i\mathrm{Ball}^{d,p}_R$. Furthermore on this multidimensional strip
    $\sup |\varphi(\cos(\cdot))| = \| \varphi \|_{\Hd{R}}$ and
    thus we can make the exponential terms small by
    moving the surface of integration. 
    
    Let us define a translation vector by $\boldsymbol{\kappa} \in \mathbb R^d$, $\|\boldsymbol\kappa\|_p = 1$ by 
        \[\kappa_j = \begin{cases}\sigma_j \cdot (\sgn k_j) \cdot
            \left|k_j\right|\rule{0pt}{12pt}^{\frac1{p-1}} \cdot
            \|\k\|^{-\frac{1}{p-1}}_q,& p > 1\\
		 \sigma_j \cdot ( \sgn k_j), & p = 1;\end{cases} \qquad j = 1,2, \ldots, d.
     \] 
    We want to
    translate our surface of integration by $\boldsymbol \kappa R$ in imaginary
    direction
    \begin{equation}
    \check \varphi_\k = 2^{-d} \sum_{\sigma \in \{-1,1\}^d}
    \frac{t_\k}{(2\pi)^d} \int_{[0,2\pi]^d} e^{i \sum_j \sigma_j k_j (\theta_j +
    i \kappa_j R)}\, \varphi(\cos(\btheta + i\boldsymbol{\kappa}R))\,
    \d\btheta. 
    \label{eq:Tphi2}
    \end{equation}
		We can compute, using the choice of~$\kappa_j$, 
		\begin{align*} 
      \left| e^{i \sum_j \sigma_j k_j (\theta_j + i \kappa_j R)} \right| &=
      \left|e^{ - \sum_j \sigma_j k_j \kappa_j R }\right| \\&= 
      \exp\left(-R \cdot  \|\k\|_q^{-\frac{1}{p-1}}\sum_{j=1}^d
      |k_j|^{\frac1{p-1}+1} \right) \\&
      = \exp\left(-R\cdot \|\k\|_q^{q-\frac{1}{p-1}}\right) = \exp(-
      R\|\k\|_q).
  \end{align*}
    Therefore we get an upper bound on $|\check \varphi_k|$ from~\eqref{eq:Tphi2} 
    \begin{align} 
      |\check \varphi_\k| \leq 2^{-d} \sum_{\sigma \in \{-1,1\}^d}
    \frac{t_\k}{(2\pi)^d} \int_{[0,2\pi]^d}  \exp(- R\|\k\|_q) \cdot  \|\varphi\|_{\Hd{R}}
      = t_\k e^{- R \| \k\|_q} \|\varphi\|_{\Hd{R}}. 
      \label{eq:ChebyCoeffBound} 
    \end{align} 
    Substituting this into \eqref{eq:ProjErrorAsCoefs} and taking into account
    $|t_\k |< 2^d$ we get
		\begin{align*} \| (\idop - \P_K) \varphi \|_{\Hd{\rho}}   &\leq
            \sum_{\|\k\|_\infty \geq K} 2^d \cdot 2e^{-(R-\rho) \| \k\|_q} \|\varphi\|_{\Hd{R}}
		\leq \sum_{\|\k\|_\infty \geq K} 2^{d+1} e^{-(R-\rho) \| \k\|_\infty}\|\varphi\|_{\Hd{R}}\\
		&= 2^{d+1} \|\varphi\|_{\Hd{R}}\sum_{\hat k = K}^\infty \# \{\k \in \mathbb N^d \mid
        \|\k\|_\infty = \hat k \}\cdot e^{-(R-\rho) \hat k}.
	 \end{align*}
	All that remains is to obtain a bound on the sum. Counting vertices in the
    hypercube we get $\# \{\k \in \mathbb N^d \mid \|\k\|_\infty = 
    \hat k \} = (\hat k+1)^d - \hat k^d$.  When $d=1$ this just leads to a geometric series.
    Otherwise by straightforward computation, since $d \in \mathbb N$ and $d \ge 2$
	\begin{align} 
	\sum_{\hat k = K}^\infty ((\hat k+1)^d - \hat k^d) e^{- (R-\rho) \hat k}
    &= \sum_{\hat k=K}^\infty \int_{\hat k}^{\hat k+1} d \cdot u^{d-1} e^{-
    (R-\rho) \hat k}\, \d u
    \leq \sum_{\hat k=K}^\infty \int_{\hat k}^{\hat k+1} d \cdot u^{d-1} e^{-
    (R-\rho)
    (u-1) }\, \d u \notag \\ &
    = d \cdot e^{R-\rho} (R-\rho)^{1-d}  \int_{K}^\infty ((R-\rho)u)^{d-1}
    e^{-(R-\rho) u}
    \d u \notag  \\ &
    = d \cdot e^{R-\rho} (R-\rho)^{-d}  \int_{K(R-\rho)}^\infty t^{d-1} e^{-t}
    \d t \notag \\ &
   	= d \cdot e^{R-\rho} (R-\rho)^{-d} \cdot \Gamma(d;K(R-\rho)).
    \label{eq:sumub0} 
	\end{align}
    It remains to get an upper bound on the incomplete Gamma function. Recall that
    for $d=2$ we get $\Gamma(2;K(R-\rho)) = e^{-K(R-\rho)}(1+K(R-\rho))$ and
    for $d = 3$ we have $\Gamma(3;K(R-\rho)) = e^{-K(R-\rho)}(2+ 2 K(R-\rho) + 
    K^2 (R-\rho)^2)$. 

    Therefore, in the case when $d=2$,
	$$
    \sum_{\hat k = K}^\infty ((\hat k+1)^2 - \hat k^2) e^{- (R-\rho) \hat k} \le 
    2 \cdot e^{-(K-1)(R-\rho)}\frac{1+K(R-\rho)}{(R-\rho)^{2}} 
    $$
    and thus 
    $$
    \| (\idop - \P_K) \varphi \|_{\mathcal H^{2,p}_\rho} \leq 
    16 \cdot e^{-(K-1)(R-\rho)}\frac{1+K(R-\rho)}{(R-\rho)^{2}}\|\varphi\|_{\Hd{R}}. 
    $$

    Similarly, if $d=3$, then 
	$$
    \sum_{\hat k = K}^\infty ((\hat k+1)^2 - \hat k^2) e^{- (R-\rho) \hat k} \le 
    3 \cdot e^{-(K-1)(R-\rho)}\frac{2+ 2 K(R-\rho) + 
    K^2 (R-\rho)^2}{(R-\rho)^{3}} 
    $$
    and hence
    $$
    \| (\idop - \P_K) \varphi \|_{\mathcal H^{3,p}_\rho} \leq 
    48 \cdot  
    e^{-(K-1)(R-\rho)}\frac{2+ 2 K(R-\rho) + K^2 (R-\rho)^2}{(R-\rho)^{3}}
    \|\varphi\|_{\Hd{R}}.
    $$
    For $d \ge 4$ we can use the bounds by Pinelis~\cite{P20}. 
    In particular, \cite[Theorem 1.1]{P20} gives 
    $$
    \Gamma(d,K(R-\rho)) \le G_d(K (R-\rho)),
    $$
    where
    $$
    G_d(x) : = \frac{(x+ \gamma)^d - x^d}{d \gamma} e^{-x} \ \mbox{ with } \
    \gamma := \sqrt[d-1]{\Gamma(d+1)}.
    $$
    By the Stirling's inequality, 
    $$
    \sqrt{2\pi d} \cdot d^d \cdot e^{\frac1{12d+1}-d} < 
    \Gamma(d+1) = d! < \sqrt{2\pi d} \cdot d^d \cdot e^{\frac1{12d}-d}
    $$ 
    therefore $\frac{\sqrt{2\pi}}{e^{2}} d \le \gamma \le
    \frac{\sqrt{2\pi}}e d$ and hence provided $d < K(R-\rho)$ 
    \begin{multline*}
    \Gamma(d,K(R-\rho)) \le
    \frac{1}{\sqrt{2\pi}d^2}\Bigl(\Bigl(K(R-\rho)+\frac{\sqrt{2\pi}}{e}d\Bigr)^d -
    (K(R-\rho))^d \Bigr) e^{2-K(R-\rho)} \\ \le
    d K^{d-1} (R-\rho)^{d-1} e^{1-K(R-\rho)}.
    \end{multline*}
    Substituting into~\eqref{eq:sumub0} we get
    \begin{align*}
	\sum_{\hat k = K}^\infty ((\hat k+1)^d - \hat k^d) e^{- (R-\rho) \hat k} &\le
    d  e^{R-\rho} (R-\rho)^{-d} \cdot     d K^{d-1} (R-\rho)^{d-1} e^{1-K(R-\rho)} \\&= 
    \frac{d^2 K^{d-1}}{R-\rho}  \cdot e^{1-(K-1)(R-\rho)}
    \end{align*}
    Summing up for $d \ge 4 $ we have a bound
	\[  
    \| (\idop - \P_K) \varphi \|_{\Hd{\rho}} \leq 
    \frac{2^d d^2 }{R-\rho}  K^{d-1} e^{1-(K-1)(R-\rho)}
    \|\varphi\|_{\Hd{R}}.
    \]
	as required.
	
	Finally, the~$O(e^{-R\|\k\|})$ decay of Chebyshev coefficients we obtain
    in~\eqref{eq:ChebyCoeffBound} implies uniform convergence of the Chebyshev 
    series on~$\Bd{\rho}$, justifying its validity. 
	\end{proof}

\section{Min-max method}\label{s:minmax}
 
In this section we explain how to use the finite rank
operator $\P_K \L$ to get estimates on the spectral radius of~$\mathcal
L$. More precisely, we take an estimate $\varphi$ of the leading eigenvector of~$\P_K \L |_{\im
\P_K}$ to construct a function 
to which we can apply the min-max method of~\cite{PV22}, modified to use our
rigorous Chebyshev--Lagrange discretisation results. 

We will prove two theorems in this section. Theorem~\ref{thm:PosOp} is a simple
modification of \cite[Lemma~3.1]{PV22}, which we do not use explicitly, but it
illustrates the way our method works and may be useful in other settings.
Theorem~\ref{t:PositiveOperator2} adds to this idea an a priori bound on
$\tfrac{\d \L_s}{\d s}$ (and hence the derivative of the spectral radius),
meaning that an upper and lower bound for the dimension that can be obtained using only
one evaluation of the transfer operator per step, as opposed to two, used in bisection
method. 

Note that in Definition~\ref{eq:hdNorm} of the norm on $\mathcal H_r^{p,d}$
we may set $r=0$, which corresponds to the supremum norm on the cube
$$
\|\varphi\|_{\mathcal H_0^{d,p} } = \sup_{\mathbf x \in [-1,1]^d}
|\varphi(\mathbf x)|. 
$$

In what follows, $\rho(\L)$ stands for the spectral radius of an operator~$\L$
on~$\Hd{r}$.
	\begin{theorem}\label{thm:PosOp}
	Suppose that for some $r > 0$, the transfer operator $\L$ given
        by~\eqref{eq:tropgen} has a spectral gap in $\Hd{r}$ and preserves positive functions on $[-1,1]^d$.
        Let $\lambda > 0$ be given. If there exist a function $\varphi \in
        \Hd{r}$ positive on $[-1,1]^d$ and a number~$K$ such that 
		\begin{enumerate}[(a)]
			\item 
			If $\sup_{\x \in [-1,1]^d} (\P_K \L \varphi - \lambda \varphi)(\x) \leq
            - \| \L \varphi- \P_K \L \varphi\|_{\Hd{0}}$,
			then $\rho(\L) \leq \lambda$; \label{thm:PosOp:partA}
			\item 
			If $\inf_{\x \in [-1,1]^d} (\P_K \L \varphi - \lambda \varphi)(\x) \geq
            \| \L \varphi- \P_K \L \varphi\|_{\Hd{0}}            $,
			then $\rho(\L) \geq \lambda$. \label{thm:PosOp:partB} 
		\end{enumerate}
	\end{theorem}

    \begin{proof}
    The argument below essentially repeats the proof of~\cite[Lemma~3.1]{PV22}. 
    Assume that \hyperref[thm:PosOp:partA]{\ref{thm:PosOp:partA}} holds. We may
    write
    \begin{multline*}
     \sup_{\x \in [-1,1]^d} (\L \varphi - \lambda \varphi)(\x)  \leq \sup_{\x
    \in [-1,1]^d} (\P_K \L \varphi - \lambda \varphi)(\x) + \sup_{\x \in
    [-1,1]^d} (\L \varphi - \P_K \L \varphi )(\x) \\  
    \leq - \|\P_K \L \varphi - \lambda \varphi\|_{\mathcal H_0^{d,p}} + 
    \sup_{\x \in [-1,1]^d} (\L \varphi - \P_K \L \varphi )(\x) \le 0.
    \end{multline*}
    Thus, $(\L \varphi)(\x) \leq \lambda \varphi(\x)$ for every $\x \in
    [-1,1]^d$.  Since $\L$ is positive, iterating the inequality we get $(\L^n \varphi)(\x) \leq
    \lambda^n \varphi(\x)$. 
    Since $\mathcal L$ has a spectral gap, there exists a positive eigenfunction
    $h \leq \varphi$ corresponding to $\rho(\mathcal L_s)$. Therefore we have
     $\rho(\L) = \limsup_{n\to\infty} (\L^n h(\x))^{1/n} \leq \limsup_{n\to\infty} (\L^n \varphi(\x))^{1/n} \leq \lambda$. 

    The proof of~\hyperref[thm:PosOp:partB]{\ref{thm:PosOp:partB}} is completely  analogous. 
    \end{proof}

In practice (see Section~\ref{s:algorithm}), we shall choose $\varphi$ to be an
informed but non-rigorous best guess of the top eigenvector of $\P_K \L$, for
instance via a non-validated computation. This best guess will be of a reasonable norm in
some appropriate Hardy space $\Hd{r}$, and so we can apply Theorem~\ref{t:Cheby}
to estimate the norm $\| \L \varphi- \P_K \L \varphi\|_{\Hd{0}}$. 

However in our considerations we are not really interested in computing the
leading eigenvalue for the Apollonian transfer operator $\mathcal A_s$, but rather in finding
a parameter value $s$ for which \emph{the spectral radius} of the  transfer
operator~\eqref{eq:ApTrOp} has a given value $\rho(\mathcal A_s) = 1$. At first sight, the above method
would suggest that $\P_K \mathcal A_s \varphi$ should be computed for at least two different values of~$s$ 
(for an upper and lower bound). 

Fortunately, in dimension estimation problems for uniformly contracting IFS such as
the induced system we consider here, we can often exploit monotonicity of the positive
operators with respect to the parameter. 

To formalise the idea let us introduce a parameter into the generic transfer
operator~\eqref{eq:tropgen} by taking weight
functions~$w_{j,s}$, that depend differentiably on the parameter $s \in (a,b)$:
\begin{equation}
    \label{eq:tropPar}
\mathcal L_s \varphi := \sum_j w_{j,s} \cdot (\varphi \circ v_j)
\end{equation}
In this case, for $s>s'$ for every $x \in [-1,1]^d$ we can bound $\L_s \varphi(x) \leq \L_s' \varphi(x) +
c(s - s') \varphi(x)$ for some constant $c$, and vice versa. 
  
The next Theorem allows us to make additional estimates so that at the
verification step (see \S\ref{s:algorithm}) we only need to compute $\mathcal P_K
\mathcal L_s \varphi$ for a single value of~$s$.  
%
The resulting bounds will be looser
than optimal by a constant factor, which is tolerable given that we are saving a
factor of at least two in computational expense (which in turn allows us to
increase $N$, $K$, etc. and obtain $10$--$15\%$ more digits in our final
estimate). 

	\begin{theorem}\label{t:PositiveOperator2}
    Suppose that there exist $r>0$ such that for all $s \in [a,b]$ the parameterised transfer
    operator~\eqref{eq:tropPar} 
    has a spectral gap in $\Hd{r}$ and preserves positive functions on $[-1,1]^d$. 
    Assume additionally that the operator 
    $$
    -\tfrac{\d}{\d s} \mathcal L_s \varphi := - \sum_j (\tfrac{\d}{\d s} w_{j,s} ) \cdot (\varphi \circ v_j)
    $$
   is also a positive operator with the property that there exist constants
   $D^{\pm}>0$ such that for every bounded $\psi > 0$ and all $\x \in [-1,1]^d$ 
	\begin{equation} -D^- \sup \psi \leq  \tfrac{\d}{\d s} (\L_s \psi)(\x) \leq
        -D^+ \inf \psi < 0.
        \label{eq:LinearResponse}
    \end{equation} 
    Let the projection operator $\P_K$ and parameter $s_0 \in [a,b]$ be given. 
    Assume that a function $\varphi \in \Hd{r}$, bounded on $[-1,1]^d$ with $ 0 <\varphi^- \leq \varphi \leq
    \varphi^+$ for some $\varphi^{\pm} \in \mathbb R^+$ and a number $\lambda >0$  
    satisfies the double inequality $ \varepsilon^- < ( \P_K \L_{s_0} \varphi - \lambda \varphi)(\x)
    < \varepsilon^+$ for some $\varepsilon^{\pm} \in \mathbb R$ and for all $\x
    \in [-1,1]^d$.     Define
    $\mathrm{err}$ to be the approximation error $\mathrm{err}: =  \| \L_{s_0} \varphi-
    \P_K \L_{s_0} \varphi\|_{\Hd{0}}$. 

    Then the unique parameter value~$s_* \in [a,b]$ for which
    $\rho(\L_{s_*}) = \lambda$
    satisfies the double inequality	
    \[ s_0 + \frac{\varepsilon^- - \mathrm{err}}{D^-  \varphi^+} <
    s_* < s_0 + \frac{\varepsilon^+ + \mathrm{err}}{D^+  \varphi^-}. \]
 \end{theorem}

\begin{proof}
    The inequality~\eqref{eq:LinearResponse} implies that the spectral radius is
    decreasing monotonically in~$s$. Indeed, differentiating  $\mathcal L_s 
    \psi = \mu_s \psi $ and taking into account $\psi > 0$ we
    conclude that $\frac{ \d}{\d s} \mu_s < 0$. 

    It follows from Theorem~\ref{thm:PosOp} that if for $\mu \in \mathbb R$ there exists a positive
    function~$\psi \in \Hd{r}$ such that $(\mathcal{L}_s \psi)(\x) < \mu
    \psi(\x)$ for all $\x \in [-1,1]^d$, then $\rho(\mathcal L_s) < \mu$.
    Hence if for some $s^\prime$ we have that $\rho(\mathcal L_{s^\prime}) =
    \mu > \rho(\L_s)$ then $s^\prime < s$. 

    Therefore in order to show that $s_* < s_0 + \frac{\varepsilon^+ +
    \mathrm{err}}{D^+  \varphi^-}$ it is sufficient to show that for 
    \[s_+ := s_0 + \frac{\varepsilon^+ + \mathrm{err}}
    {D^+ \varphi^-} \] 
    we have $\rho(\L_{s_+}) < \lambda$, or $\mathcal L_{s_+} \varphi (\x) < \lambda \varphi (\x) $ for all $\x
    \in [-1,1]^d$. 

    Integrating \eqref{eq:LinearResponse} with $\psi = \varphi$ we
    have that \[ (\L_s \varphi)(\x) - (\L_{s_0} \varphi)(\x) = \int_{s_0}^s
    \tfrac{\d}{\d t} (\L_{t} \varphi)(\x)\,\d t \leq \int_{s_0}^s
    -D^+ \varphi^- \,\d  t, \] 
	so for $s > s_0$,
    \begin{equation}
        \label{eq:trop2eq1}
	(\L_s \varphi)(\x) - (\L_{s_0} \varphi)(\x) \leq -D^+ \varphi^- (s - s_0). 
    \end{equation}
	Consequently, if we set  
    $s_+ = s_0 + \frac{\varepsilon^+ + \mathrm{err}}
    {D^+ \varphi^-} > s_0$, 
    we have by~\eqref{eq:trop2eq1}
	\begin{align*} 
        (\L_{s_+} \varphi)(\x) &- \lambda \varphi(\x) \leq (\L_{s_0}
        \varphi)(\x)-D^+  \varphi^- (s_+ - s_0)- \lambda \varphi(\x)\\
        &  =((\L_{s_0} \varphi)(\x) - (P_K \L_{s_0} \varphi)(\x)) + ( (P_K
        \L_{s_0} \varphi)(\x) - \lambda \varphi(\x) ) - D^+  \varphi^-
      \frac{\varepsilon^+ + \mathrm{err} }{D^+  \varphi^-}\\ 
      & \leq \mathrm{err} + \varepsilon^+ - D^+  \varphi^- \frac{\varepsilon^+ + \mathrm{err} }{D^+  \varphi^-}
       \leq 0. 
	\end{align*}
    The lower bound proceeds similarly. 
\end{proof}

In Theorems~\ref{thm:PosOp} and \ref{t:PositiveOperator2}, we need to compute
suprema and infima of polynomials. The following results make obtaining
reasonable bounds very easy: 
\begin{proposition}\label{p:BoundByCheby}
	Suppose $\varphi = \sum_{\k \in \N^d} \check{\varphi}_\k T_\k \in \Hd{\rho}$
    for some $\rho > 0$ (for example, if $\varphi$ is a polynomial). Then  
	\[ \sup_{\x \in [-1,1]^d} \varphi(\x) \leq \check{\varphi}_{\boldsymbol{0}}
    + \sum_{\k \in \N^d \backslash \{0\}} |\check{\varphi}_\k| \] 
	and
		\[ \inf_{\x \in [-1,1]^d} \varphi(\x) \geq \check{\varphi}_{\boldsymbol{0}} - \sum_{\k \in \N^d \backslash \{0\}} |\check{\varphi}_\k| \]
\end{proposition}
\begin{proof}
	Because $\varphi \in \Hd{\rho}$, by Proposition \ref{p:ProjectionNorms} we
    know that the sum of Chebyshev coefficients converges uniformly on
    $[-1,1]^d$. Then the result holds because $T_{\boldsymbol{0}} = 1$ and
    $|T_{\k}| \leq 1$ on $[-1,1]^d$ for all $\k \in \N^d$. 
\end{proof}

Assume now that we computed non-rigorously a value $s_0$ and function $\varphi$ so that
$\varphi$ approximates an eigenfunction of $\L_{s_0}$, and the eigenvalue of
$\L_{s_0}$ is very close to $1$ (i.e. zero pressure), we will then rigorously
compute $\P_K \L_{s_0} \varphi - \varphi$, and apply
Proposition~\ref{p:BoundByCheby} to obtain rigorous upper and lower bounds on
$\P_K \L_{s_0} \varphi - \varphi$. 
With this information, Theorem~\ref{t:PositiveOperator2} then provides both an
upper bound $s_+$ and lower bound $s_-$ to rigorously enclose the true dimension
$s \in [s_-, s_+]$, giving us our main theorem (Theorem~\ref{t:main}). 

Given $s \in (1.2,1.4)$ and $K$ we obtain a suitable approximation $\varphi$ to the
eigenfunction of the Apollonian transfer operator~$\mathcal A_{s}$ from the eigenvector of the matrix $P_K
\mathcal A_{s}\bigl|_{\mathtt{im} P_k}$. It is therefore necessary to be able compute the 
matrix itself effectively and efficiently. This means we need to evaluate $\mathcal A_{s}
\ell_{\k,K}(x_{\mathbf j,K})$ for $\k \in \{(k_1,k_2) \mid k_1, k_2 \le K\}$,
where $\ell_{k,K}$ are the  Chebyshev--Lagrange polynomials~\eqref{eq:LagrangePoly} and $x_{\mathbf
j,K}$ are Chebyshev nodes~\eqref{eq:ChebyshevNodes}. In the next section we
explain how to evaluate $\mathcal A_{s} \varphi (x)$ given an analytic function~$\varphi$. 

\section{The Euler--Maclaurin formula}
\label{s:pointapprox}

A major obstacle to accurate computation of the Hausdorff dimension of the
Apollonian circle packing is the fact that the transfer operator~$\mathcal{A}_s$
 defined by~\eqref{eq:ApTrOp} as an infinite sum with the tail 
that decays at a relatively slow polynomial rate (since $J_n(x,y) = O(n^{-2})$ as
$n\to\infty$). However, the analyticity of $J_n^\pm$ and $G_n^\pm$ with respect to $n$
allows us to apply a kind of sum acceleration, the so-called
Euler--Maclaurin formula. 

The idea is as follows. We will aim to approximate the infinite sum \[\sum_{n=0}^\infty \psi(n).\]
We first choose~$N>0$ and compute the sum of the terms exactly up to $n = N-1$. Then, we want to
approximate the tail by an integral: 
\begin{equation}
    \label{eq:EMIntegral}
\sum_{n=N}^\infty \psi_n \simeq I_\psi(N) := \int_N^\infty \psi(t)\,\d t.
\end{equation}
There is obviously a substantial error here. We can iteratively correct the
error, let us say with $L+1$ terms for $L\geq 0$. We choose the first to be $\tfrac{1}{2}
\psi(N)$; the rest are odd derivatives of $\psi$ at $N$, with  
coefficients coming from Bernoulli numbers~$B_l$ 
\begin{equation}
    \label{eq:EMSumOfDerivatives}
D_\psi(L,N) = \sum_{l=1}^L \frac{B_{2l}}{(2l)!} \left.\frac{\partial^{2l-1}
\psi(t)}{\partial^{2l-1} t}\right|_{t=N}.
\end{equation}
This correction of the error is not perfect, and leaves a remainder~$R_\psi(L,N)$, which, provided the appropriate
choices of~$L$ and~$N$ are made, will be very small and can be bounded with some extra information about the regularity of~$\psi$. 

Putting these terms altogether,  
\begin{equation}
\sum_{n=0}^\infty \psi(n) = \sum_{n=0}^{N-1} \psi(n) + \tfrac{1}{2} \psi(N) +
 I_\psi(N) + D_\psi(L,N) + R_\psi(L,N).
\label{eq:EMFormula} 
\end{equation}

Although~\eqref{eq:EMFormula} holds true for any
$\psi\in C^{2L}(\mathbb R)$, a holomorphic extension of the summand function helps us to
obtain a good estimate on the error term, as well as 
very effective computable approximations of the other terms $I_\psi(N)$ and $
D_\psi(L,N)$ involving only pointwise evaluations of $\psi(n)$ at different
(complex) values of~$n$. 

We proceed through the error analysis in reverse order of the appearance of the terms.

\subsection{Remainder term $R(L,N)$}
\label{ss:EMRemainder}

\begin{proposition}
  \label{p:EMRemainderBound}
  Assume that on the right half-plane $\Re (z) \geq \nu$ a function~$\psi$ is
  analytic and bounded $|\psi(z)|\leq C$. 
  Then for all real $N > \nu$ and for all integer~$L\ge1$, 
  \begin{equation}
  |R_\psi(L,N)| \leq \frac{(2L+1)!C}{L (2\pi)^{2L+1} (N-\nu)^{2L}}.  
  \label{eq:Rtop}
  \end{equation}
\end{proposition}
Note that we do not require the function~$\psi$ to be analytic at infinity.
\begin{proof}
	A classical result states that the remainder term can be written using Bernoulli
  polynomials $B_n(x)$
	\begin{equation*}
	R_\psi(L,N) = \int_N^\infty \frac{1}{(2L+1)!}
  \frac{\partial^{2L+1}}{\partial t^{2L+1}} \psi(t) B_{2L+1}(\{t\}) \,\d t. 
	\end{equation*}
  where~$\{t\}$ stands for the
  fractional part of~$t$. Using Lehmer's bound~\cite{L40} for the Bernoulli polynomial
   $|B_{2L+1}(x)| \leq 2(2\pi)^{-(2L+1)}(2L+1)!$, we may write
	\begin{equation} 
        |R_\psi(L,N)| \leq \frac{2}{(2\pi)^{2L+1}}
        \int_N^\infty \left|\frac{\partial^{2L+1}}{\partial t^{2L+1}}
        \psi(t)\right| \,\d t. 
        \label{eq:EMRemainderStandard} 
    \end{equation} 
	We can estimate the derivative at~$t$ via an integral along a complex circle
    centred at~$t$ of radius $|\nu - t|$ using the Cauchy's formula. For $\Re(t) > \nu$
    this gives 
	\begin{align*}
        \left|\frac{\partial^{2L+1}}{\partial t^{2L+1}} \psi(t)\right| &\leq
        \frac{(2L+1)!}{2\pi} \int_0^{2\pi} |\psi(t+|\nu-t|e^{i\vartheta})|\, |t- \nu
        |^{-(2L+1)}\,\d \vartheta \\ 
	    &\leq (2L+1)! \, |t-\nu|^{-(2L+1)}C. 
    \end{align*}
	Substituting this into \eqref{eq:EMRemainderStandard}, we get
	\begin{align*} 
        |R_\psi(L,N)| &\leq \frac{2}{(2\pi)^{2L+1}} \int_N^\infty (2L+1)! \,
        |t-\nu|^{-(2L+1)} C \,\d t\\
         &=  \frac{C}{2\pi L} (2L+1)! \cdot (2\pi \cdot |N-\nu|)^{-2L} 
    \end{align*}
	 as required.
\end{proof}

\begin{remark}\label{r:EMKNScaling}
	For fixed $N$, the choice of $L$ minimising the
    Proposition~\ref{p:EMRemainderBound} bound on $R(L,N)$ is such that $2L+1
    \approx 2\pi (N-\nu)$. This gives exponentially small error with
	\[ |R(L,N)| = \exp\left(-2\pi (N-\nu) + O(\log (N-\nu))\right);  \]
  by the virtue of the Stirling formula $\log (n!) = n \log n - n + O(\log n)$.  
\end{remark} 

\subsection{Sum of derivatives term $D(L,N)$}
\label{ss:EMSumOfDerivatives}

To numerically estimate~$D(L,N)$ we will need to be able to accurately
estimate high-order derivatives of an analytic function at a given
point\footnote{In the setting we consider
this is $\psi(j): = w_j(z) \varphi(v_j(z))$ from~\eqref{eq:tropgen}  at $j=N$.}.

Fortunately the Taylor coefficients of a function $\psi$ at a point $z_0$ are close to the positive order Fourier
coefficients of $\psi$ on a circle $z_0
+ \tau e^{i\theta}$, $0 \le \theta < 2\pi$. We can therefore estimate the derivatives of~$\psi$ by
taking a discrete Fourier transform of the function on the 
circle. The error bounds are given in the following lemma. 
\begin{lemma}
  \label{l:DiscreteTaylorDerivative}
  Suppose that on a disc $\{ z \in \C \mid |z-z_0|< \sigma \}$, a function $\psi$ is
    holomorphic and bounded $|\psi(z)|<C$. Then for all $k \in \N$, 
	\begin{equation} 
    |\psi^{(k)}(z_0)| \leq \frac{k!}{\sigma^{k}}  C.
        \label{eq:DiscreteTaylorDerivativeBound}
    \end{equation}
 Furthermore, the $k$'th derivative at~$z_0$ can be approximated using evaluations at $M
 > k $ equidistant points on a circle of radius $\tau < \sigma$: 
	\begin{equation} 
    \left|\psi^{(k)}(z_0) - \frac{k!}{M\tau^{k} } \sum_{m=1}^{M} e^{-\pi i
    k\frac{2m-1}M} \psi\left(z_0 + \tau e^{\pi i \frac{2m-1}M}\right) \right| \leq
   \frac{ C  \tau^M k!}{ \sigma^k (\sigma^M-\tau^M)}. 
        \label{eq:DiscreteTaylorError}
    \end{equation}
\end{lemma}
\begin{proof}
	The first bound follows immediately from the Cauchy's integral theorem. 
	
	For the second bound, we essentially compute an aliasing error for the
  discrete Fourier transform performed along the circle of the smaller
  radius~$\tau< \sigma$ centred at~$z_0$. Expanding~$\psi$ as a Taylor series
  about~$z_0$, we get
  	\begin{align*} 
        \psi\left(z_0 + \tau e^{\pi i \frac{2m-1}M}\right) &= \sum_{j=0}^\infty
        \frac{1}{j!} \psi^{(j)}(z_0) \tau^j e^{\pi i j \frac{2m-1}M}. 
	\end{align*}
    Substituting to the sum on the left hand side of~\eqref{eq:DiscreteTaylorError} and swapping the order of
    summation, we obtain: 
   $$     
        \frac{1}{M} \sum_{m=1}^M e^{-\pi i k\frac{2m-1}M} 
        \psi\left(z_0 + \tau e^{2\pi i \frac{2m-1}M}\right)
   = \sum_{j=0}^\infty \frac{\tau^j}{j!} \psi^{(j)}(z_0) \frac1M \sum_{m=1}^M
   e^{\pi i (j-k)\frac{2m-1}M}.
   $$
   Applying the geometric series identity to the sum in~$m$, 
   $$
   \frac{1}{M} \sum_{m=1}^M e^{\pi i (j-k)\frac{2m-1}M} = \begin{cases} (-1)^j &
       \mbox{ if } j = k + Mk^\prime, \; k^\prime \in \mathbb Z, \\
       0 & \mbox{ otherwise } .
   \end{cases}
   $$
   Therefore 
   $$
   \frac{1}{M} \sum_{m=1}^M e^{-\pi i k\frac{2m-1}M} 
        \psi\left(z_0 + \tau e^{2\pi i \frac{2m-1}M}\right)
   = \sum_{j=0}^\infty \frac{(-1)^j \tau^{k + M j}}{(k + M j)!} \psi^{(k + M j)}(z_0) .
   $$
	Then, the left-hand side of \eqref{eq:DiscreteTaylorError} becomes
	\begin{align*} 
    \left| \psi^{(k)}(z_0) - k! \sum_{j=0}^\infty (-1)^j \tau^{M j} \frac{\psi^{(k
    +Mj)}(z_0)}{(k +M j)!}\right| =  \left| k! \sum_{ j=1}^\infty (-1)^{ j}
    \tau^{M j} \frac{\psi^{(k +M j)}(z_0)}{( k +M j)!}     \right| 
  \end{align*} 
	which is using the first bound \eqref{eq:DiscreteTaylorDerivativeBound} can be bounded by
\begin{equation*}
   \left| k! \sum_{ j=1}^\infty
    \tau^{M j} \frac{\psi^{(k+M j)}(z_0)}{(k+M j)!} (-1)^{ j}
    \right| 
    \le k!  \sum_{j=1}^\infty \tau^{M j} C \sigma^{-(k+M j)} \leq  
  \frac{C \tau^M k!}{\sigma^k(\sigma^M-\tau^M)}  
\end{equation*}
	as required.
\end{proof}

We can now apply this general result to compute~$D_\psi(L,N)$ very efficiently.

\begin{proposition}
  \label{p:EMDerivativeError}
Suppose that for some $\nu$, $C$ we have that $|\psi(z)|\leq C$ on the right
half-plane $\{z \mid \Re (z) \geq \nu\}$. Choose $M \geq 2L$ and consider the circle centred
at~$N$ of radius $\tau_{N}: = \frac{N-\nu}e$, and pick evaluation points 
	\begin{align} 
    z_{m} &= N + \tau_{N} e^{2\pi i \tfrac{2m-1}{2M}}, \quad m = 1, \ldots, M
    \label{eq:zmdef}\\
    \intertext{together with coefficients}
    c_m &= \sum_{l=1}^L \frac{B_{2l}}{2l} \tau_{N}^{-(2l-1)} e^{-2\pi i
    \tfrac{(2l-1)(2m-1)}{2M}}, \quad m = 1, \ldots M.  
    \label{eq:cmdef}
	\end{align}
	Then provided $L$ and $N$ satisfy $2e\pi (N-\nu) > 2L$, we have
 \begin{equation} 
     \left| D_\psi(L,N) - \frac1M \sum_{m=1}^M c_m \psi(z_{m}) \right|
  \leq  \frac{\pi^2 e}{6}\cdot \frac{ N-\nu}{(\tfrac{2 \pi e (N-\nu) }{2L-1})^2 -
  1} \cdot \frac{C}{e^M - 1}.
  \label{eq:EMDerr}
  \end{equation}
\end{proposition}
\begin{proof}
	Using \eqref{eq:EMSumOfDerivatives}, the definition of $D_\psi(L,N)$, we
    write
  \begin{align*}
  \Bigl| & D_\psi(L,N) - \frac{1}M \sum_{m=1}^{M} c_m \psi(z_{m}) \Bigr| \\ &=
 \left| \sum_{l=1}^L \frac{B_{2l}}{(2l)!} \frac{\partial^{2l-1}
 \psi(t)}{\partial^{2l-1} t}\Bigl|_{t=N} - \frac{1}M \sum_{m=1}^{M} c_m
 \psi(z_{m}) \right| \\ &= 
\left| \sum_{l=1}^L \frac{B_{2l}}{(2l)!} \frac{\partial^{2l-1}
 \psi(t)}{\partial^{2l-1} t}\Bigl|_{t=N} - \frac{1}M \sum_{m=1}^{M} 
\sum_{l=1}^L \frac{B_{2l}}{2l} \tau_{N}^{-(2l-1)} e^{-2\pi i
\tfrac{(2l-1)(2m-1)}{2M}} 
\psi\Bigl(N + \tau_{N} e^{2\pi i \tfrac{2m-1}{2M}}\Bigr) \right| \\ &\leq
\sum_{l=1}^L \frac{|B_{2l}|}{(2l)!} \cdot \left|\frac{\partial^{2l-1}
\psi(t)}{\partial^{2l-1} t}\Bigl|_{t=N} - \frac{(2l-1)!}{M \tau_N^{(2l-1)} }
\sum_{m=1}^M e^{-2\pi i
\tfrac{(2l-1)(2m-1)}{2M}}\psi\Bigl(N + \tau_{N} e^{2\pi i
\tfrac{2m-1}{2M}}\Bigr) \right| 
\end{align*}
Applying Lemma~\ref{l:DiscreteTaylorDerivative} with $\sigma = N -
\nu$ and $\tau = \tau_N = \sigma/e$ we can bound this by
\begin{align*} \sum_{l=1}^L \frac{|B_{2l}|}{(2l)!} \cdot
\frac{(2l-1)! \tau^M }{\sigma^{(2l-1)}(\sigma^M-\tau^M)} C = \frac{C}{e^M-1} 
\sum_{l=1}^L \frac{|B_{2l}|}{(2l)!} \cdot \frac{(2l-1)!  }{\sigma^{(2l-1)}
}. 
  \end{align*}
  Taking into account that $\frac{B_{2l}}{(2l)!} = \tfrac{(-1)^l
  2\zeta(2l)}{(2\pi)^{2l}}$ we conclude
	\begin{equation}
  \Bigl|  D_\psi(L,N) - \frac{1}M \sum_{m=1}^{M} c_m \psi(z_{M,m}) \Bigr|  \le
    \frac{C \sigma }{e^M-1} \sum_{l=1}^L \zeta(2l)
  \frac{(2l-1)!}{(2\pi\sigma)^{2l}} 
  \label{eq:Dsum}
  \end{equation}
  Since $\zeta(2l) \leq \zeta(2) = \frac{\pi^2}6$ and $(2l-1)! \leq
  (2l-1)^{2l}\cdot e^{-(2l-1)} \leq (2L-1)^{2l} e^{1-2l}$, we can, under the assumption 
  $2L-1 < 2e \pi \sigma = 2e \pi (N-\nu)$, bound the latter sum by
	\begin{equation*}
    \sum_{l=1}^L \zeta(2l) \frac{(2l-1)!}{(2\pi\sigma)^{2l}} \le
    \sum_{l=1}^L \frac{\pi^2 e}{6} \left(\frac{2L-1}{2e \pi \sigma}\right)^{2l} \leq 
    \frac{\pi^2 e}{6} \cdot \frac{1}{\left(\frac{2e\pi \sigma}{2L-1}\right)^2  - 1 }, 
     \end{equation*}
Substituting into~\eqref{eq:Dsum} we complete the proof.  	
  (Additionally it follows that the error is in fact small.)
\end{proof}
\begin{remark}
	The reason we choose $\tau = \tau_N = (N-\nu)/e$ here is that the error in
  Lemma~\ref{l:DiscreteTaylorDerivative} is improved by making $\tau$ as small
  as possible with respect to $\sigma = N-\nu$, but to avoid floating-point
  roundoff error we need the coefficients~$c_m$ given by~\eqref{eq:cmdef}
  not to grow too fast as
  $L,N\to\infty$ (using the scaling in Remark~\ref{r:EMKNScaling}), which occurs
  if $\tau$ is made too small. The factor of $1/e$ is, in this specific
  case, the smallest ratio which avoids this latter phenomenon.
\end{remark}
\begin{remark}
  For similar reasons, our condition on $\frac{L}N$ is looser than the scaling given
  in Remark~\ref{r:EMKNScaling} by a factor of~$e$. 
\end{remark}

\subsection{Integral term $I(N)$}
\label{ss:EMIntegral}

\begin{proposition}
  \label{p:EMIntegralError}
   Assume that for some $\nu < N$ there exist a function $\phi$ on the Riemann sphere analytic and
   bounded on a disk $|z|<\frac{1}\nu$  with $| \phi (z)|< C$ such
   that $\psi(z) = \frac1{z^{2s}} \phi\left(\frac1z\right)$ for some $s>\frac12$. 	
   Given $M^\prime>0$ define evaluation points 
	\begin{align}
       z_{k}^\prime &= N e^{-\pi i \tfrac{2k+1}{M^\prime}}, \quad k = 1, \ldots,
       K; \label{eq:zmpdef}
   \intertext{ and coefficients}
c_k^\prime &= N e^{2\pi i s \frac{2k+1}{M^\prime}} \sum_{p=0}^{M^\prime-1} \frac{ e^{- \pi i p 
       \tfrac{(2k+1)}{ M^\prime}}}{ p + 2s -1},\quad k = 1, \ldots, M^\prime. 
       \label{eq:cmpdef}
	\end{align} 
Then the integral can be approximated by a finite sum
	\begin{equation} 
    \left| \int_N^\infty \psi - \frac1{ M^\prime} \sum_{k=0}^{ M^\prime-1}
    c_k^\prime \psi(z_k^\prime) \right| \leq \frac{2 C N^{1-2s}}{(2s-1)(1 -
    \tfrac{\nu}{N})} \frac{1}{(\tfrac{N}{\nu})^{ M^\prime} -1}.
    \label{eq:EMIntegralErrorBound} 
  \end{equation} 
\end{proposition}

\begin{remark}
	The error \eqref{eq:EMIntegralErrorBound} is $\sim (\tfrac{\nu}{N})^{M^\prime}$. Thus
  to obtain an error of the order of $\varepsilon$, it is necessary to 
  set $ M^\prime \sim \tfrac{\log \varepsilon^{-1}}{\log (N/\nu)}$. According to Remark
  \ref{r:EMKNScaling}, $N$ should grow as $O(\log \varepsilon^{-1})$ as
  well, thus $ M^\prime$ is one of the slower-growing parameters. 
\end{remark}

\begin{proof}
	Since $\phi$ is analytic on $\left|z\right| < \frac1\nu$ it can be
  expanded in Taylor series at zero that converges absolutely in the disk 
	$$
  \phi(z) = \sum_{k=0}^\infty z^k   \frac{\phi^{(k)}(0)}{k!}.
  $$
	Since by assumption $\nu < N$ we may write
	\begin{equation*} 
    I_\psi(N) = \sum_{k=0}^\infty \int_N^\infty \frac{\phi^{(k)}(0)}{k!}
    z^{-(k+2s)}\,\d z  
	= \sum_{k=0}^\infty \frac{\phi^{(k)}(0)}{k!} \frac{1}{(k+2s-1) N^{k+2s-1}}.
	\end{equation*}
	First, we want to approximate $I_\psi(N)$ by a weighted sum of derivatives. 
  	\begin{align} 
    \Bigl| I_\psi(N) - \sum_{k=0}^{M^\prime-1}
    \frac{\phi^{(k)}(0)}{k! N^k} \frac{N^{1-2s}}{k+2s-1} \Bigr| 
    &\leq \sum_{k=M^\prime}^{\infty} \frac{|\phi^{(k)}(0)|}{k! N^k} \frac{N^{1-2s}}{(k+2s-1)} \notag 
    \intertext{by the first bound in Lemma~\ref{l:DiscreteTaylorDerivative} with
    $\sigma=\frac1\nu$,}
	& \leq C  N^{1-2s}\sum_{k=M^\prime}^\infty \frac{ \nu^k  }{(k+2s-1) N^k}
  \notag \\ 
	& \leq \frac{ C N^{1-2s}}{(M^\prime+2s-1)(1-\tfrac{\nu}{N})}
    \left(\frac{\nu}{N}\right)^{M^\prime} =:\varepsilon_1. 
  \label{eq:EMIntegralFiniteSumError} 
	\end{align}
  We next want to apply the second part of Lemma~\ref{l:DiscreteTaylorDerivative} to estimate the difference between the
  weighted sum and the Fourier expansion. 
  Note that by the choice of~$c_k^\prime$ and $z_k^\prime$,
  \begin{multline}
  \sum_{k=0}^{M^\prime-1} c_k^\prime \psi(z_{k}^\prime) = \sum_{k=0}^{M^\prime-1} c_k^\prime
  \frac1{z_k^{2s}} \psi\left( \frac1{z_k} \right) \\ = N \sum_{k=0}^{M^\prime-1} e^{2\pi i
  s \frac{2k+1}{M^\prime}} \frac{1}{N^{2s}} e^{-2\pi i s \frac{2k+1}{M^\prime} } \phi
  \left(\frac1N e^{2\pi i \frac{2k+1}{M^\prime}} \right) \sum_{p=0}^{M^\prime-1} e^{-2\pi i p
  \frac{2k+1}{M^\prime} } \frac{1}{p+2s-1} \\ = N^{1-2s} \sum_{k=0}^{M^\prime-1}
  \frac{1}{k+2s-1} \sum_{p=0}^{M^\prime-1} e^{-\pi i k \frac{2p+1}{M^\prime}}   \phi
  \left(\frac1N e^{2\pi i \frac{2p+1}{M^\prime}} \right) 
  \end{multline}
  Therefore we can rewrite the difference as follows
	\begin{align*}
        \Bigl| \sum_{k=0}^{M^\prime-1} \frac{\phi^{(k)}(0)}{k! N^k}
        \frac{N^{1-2s}}{k+2s-1} & - \frac1{M^\prime} \sum_{k=0}^{M^\prime-1} c_k^\prime \psi(z_{
        k}^\prime) \Bigr| \\ &= 
        \left| \sum_{k=0}^{M^\prime-1} \frac{N^{1-2s}}{k +2s-1 } 
        \left( \frac{\phi^{(k)}(0)}{k! N^k} - \frac1{M^\prime} \sum_{p=0}^{M^\prime-1} e^{-\pi
        i k \frac{2p+1}{M^\prime} } \phi\left(\frac 1N e^{2\pi i \frac{2p +1}
        {M^\prime} } \right)         \right) \right|. 
	\end{align*}
      Applying Lemma~\ref{l:DiscreteTaylorDerivative} with $\tau = \frac1N$ and $\sigma = \frac1\nu$ 
    \begin{align}
 \Bigl| \sum_{k=0}^{M^\prime-1} \frac{\phi^{(k)}(0)}{k! N^k} 
 \frac{N^{1-2s}}{k+2s-1}  - \frac1{M^\prime} \sum_{k=0}^{M^\prime-1} c_k^\prime \psi(z_{
        k}^\prime) \Bigr|  
        & \le \sum_{k=0}^{M^\prime-1} \frac{N^{1-2s}}{k +2s-1 } \Bigl| \frac{C 
        N^{k-M^\prime} }{\nu^{-k}(\nu^{-M^\prime} - N^{-M^\prime})} \Bigr| \notag \\ &\le \frac{C
        N^{1-2s}}{(2s-1) (1-\frac{\nu}N) }
        \frac{1}{\left(\frac{N}\nu\right)^{M^\prime}-1} = : \varepsilon_2.
        \label{eq:EMIntegralFinitePointsError}
    \end{align}
    Putting the two bounds~\eqref{eq:EMIntegralFiniteSumError}
    and~\eqref{eq:EMIntegralFinitePointsError} together, we conclude that 
    $$
    \Bigl|I_\psi(N) - \frac1{M^\prime} \sum_{k=0}^{M^\prime-1} c_k^\prime \psi(z_{
        k}^\prime) \Bigr| \le \varepsilon_1 + \varepsilon_2 \le 2 \varepsilon_2, 
    $$
    which completes the proof. 
\end{proof}
 
We can now summarize Propositions~\ref{p:EMRemainderBound},~\ref{p:EMDerivativeError}, and \ref{p:EMIntegralError} in the following Theorem,
which is one of the key results of the paper and one of the essential
ingredients in computation of dimension.  
 \begin{theorem}
     \label{t:EMgeneric}
     Under the hypotheses and in the setting of Theorem~\ref{t:Cheby}, assume
     additionally that for some $\alpha \in \R$ and every fixed $x \in
     [-1,1]^d$, $n^\alpha w_n(x)$ and $v_n(x)$ extend to analytic families for
     $n \in \hat{\mathbb C} \backslash B(0,\nu)$.

     Then for any $\varphi \in \Hd{r}$ the function~$\mathcal L_s \varphi$ can be 
     evaluated at any point in $[-1,1]^d$ with arbitrary precision by evaluating a finite sum. 
     Namely,  for any $\mathrm{Err} > 0$ there exist sequences of complex
     numbers $\{n_q\}, \{w_q\}$ of length $Q = O(\log \mathrm{Err})$
     that depend on the system only such that 
     $$
     \left| [\mathcal L_s \varphi](\cdot) - \sum_{q=1}^Q w_{q}(\cdot) \varphi(v_{n_q}(\cdot)) \right| \leq
      \|\varphi\|_{\Hd{r}}\cdot\mathrm{Err}.
     $$
 \end{theorem}
We shall omit the proof of this generic statement in favour of giving a
complete proof in the special case when $\L_s = \mathcal A_s$, which is our
main interest and for which we shall give explicit formulae for the
$\mathtt{Err}$ term.
The proof of the generic case is completely analogous.
\begin{remark} 
    In the setting of the
     Apollonian gasket, when $\varepsilon$ is small enough we will need no more
     than $Q <  2 \log \varepsilon$ terms to
     guarantee that~$\mathrm{Err} < \varepsilon$.
 \end{remark}

\section{Numerical estimates for Apollonian IFS}
In this section we explicitly compute the constants that occur in suppositions
of Propositions~\ref{p:AnalyticAtInfinity}, \ref{p:TransferBounded},
\ref{p:EMRemainderBound}, \ref{p:EMDerivativeError}, and \ref{p:EMIntegralError}
and Theorems~\ref{t:Cheby} and \ref{t:PositiveOperator2} when these are applied to the Apollonian transfer
operator~$\mathcal A_s$. We then use them combined with these theorems
to prove our main Theorem~\ref{t:main}, bounding the dimension.

\subsection{Pointwise estimates for the induced IFS}
\label{s:aprioribdd}
We start with the bounds on the maps $G_n^\pm$ that constitute the iterated
function scheme and their Jacobians, defined by~\eqref{eq:Rmaps}
and~\eqref{eq:RmapsJac}, respectively.
\begin{proposition}\label{p:APrioriConstants}
Let us choose $R_A = 1.4$ and $r_A = 0.9$ to be sizes of the ellipses, and
$\nu_A = 10$ to be the radius of analyticity in $n$. 
Then the following hold:
\begin{enumerate}[(a)]
	\item\label{res:CLBernsteinInclusion} 
        The $G_n^\pm$ map the large ellipse into the small: $G^\pm_n(\Bd[2,2]{R_A}) \subseteq \Bd[2,2]{r_A} $ for all $n \in \N$. 
	\item\label{res:CLJacobianBound} For all $(z_1,z_2) \in \Bd[2,2]{R_A}$, $n \in
        \N$ the Jacobian is bounded $ |J^\pm_n(z_1,z_2)| \leq \max\{\frac{3}{4 (n+1)}, \frac{36}{n^2}\}$.
    \item\label{res:Large} On the square $[-1,1]^2$ both statements extend to all large enough
        complex $n \in \mathbb C$. Specifically, if $|n| \geq \nu_A$, then  
         $ G^\pm_n([-1,1]^2) \subseteq \Bd[2,2]{r_A}$ and $|J^\pm_n(x,y)| \leq
         \frac{6.8}{|n|^2}$ for all $(x,y) \in [-1,1]^2$. 
\end{enumerate}
\end{proposition}
        In particular, we can choose $\nu = 10$ in
        Propositions~\ref{p:EMRemainderBound},~\ref{p:EMDerivativeError},
        and~\ref{p:EMIntegralError}. 
In what follows, $R_A=1.4$, $r_A = 0.9$, $\nu_A = 10$ are the constants. 
\begin{proof}
The argument is computer-assisted. Since the maps $G_n^\pm$ are analytic and the
Jacobians are harmonic, they obey the maximum principle. Therefore we can
establish the bounds by partitioning the boundaries of the sets into boxes and
using validated interval arithmetic on the boxes. 

To be more specific, recall that by definition~\eqref{eq:el2} the Bernstein
ellipse is
\begin{equation*}
  \mathcal E_r^{2,2} = \cos \left(\R^2 + i \cdot \{ x \in \mathbb R^2 \mid |x| < r \}  \right) \subset \C^2.
\end{equation*}
Therefore its boundary is given by
$\partial \Bd[2,2]{R_A} = \cos([0,\pi] + i \cdot \{x \in \R^2 \mid |x|=R_A\}) $.
Consequently, we divide $[0,\pi]$ and the circle of radius $R_A$ into $30$--$50$
equal intervals each (with interval arithmetic), to obtain two partitions that
we call $p_{[0,\pi]}$ and $p_{circ}$, respectively. We then take the image of those
partitions under the transformation $(p_{[0,\pi]},p_{circ}) \mapsto \cos
(p_{[0,\pi]} + ip_{circ})$ to obtain a partition of the boundary~$ \partial \Bd[2,2]{R_A}$. 

The computer code is available in the Jupyter notebook {\tt APriori.ipynb}. 
\end{proof}

With these bounds in hand, we are now ready for 
our last preparatory result, which is to 
compute pointwise evaluations of the transfer operator of the Apollonian circle
packing, i.e. $(\mathcal A_s \varphi)(z)$ at a point~$z \in \Bd[2,2]{r_A}$.

The next theorem is an adaptation of~Theorem~\ref{t:EMgeneric} to the Apollonian
transfer operator. The error $\mathrm{Err}$ is a sum of approximation errors in
Propositions~\ref{p:EMRemainderBound}--\ref{p:EMIntegralError}, i.e. the sum of
right hand sides of~\eqref{eq:Rtop},~\eqref{eq:EMDerr}
and~\eqref{eq:EMIntegralErrorBound} and depends on the constants $N,L,M$,
and~$M^\prime$ defined therein. In Table~\ref{tab:Scalings} below we recall the
meaning of these
constants and specify the most efficient choices that ensure the error is of the
order~$\varepsilon$.

\begin{table}[h]
	\caption{Approximation parameters and their respective ideal scalings against a
    desired precision~$\varepsilon$. As follows from
    Proposition~\ref{p:APrioriConstants} we may choose $\nu_A\ge 10$ and $R_A =
    1.4$. Here $a \gtrsim b$ means $a\geq b - o(\log \varepsilon)$, and $a \sim
    b $ means $a - b = o(\log \varepsilon)$.} 
	\centering
	\begin{tabular}{|c|lll|}	
		\hline
		Parameter & Description & Scaling & Reference \\
		\hline
		$K$ & Order of Chebyshev--Lagrange & $\gtrsim - \log \varepsilon / R_A$ & Th.~\ref{t:PositiveOperator2} \\
        $N$ & Starting point for Euler--Maclaurin & $\gtrsim \nu_A - \log \varepsilon / 2\pi$ & Rem.~\ref{r:EMKNScaling}\\
        $L$ & \# of derivatives in Euler--Maclaurin & $\sim [ \pi
        (N-\nu_A)+1]$ & Rem.~\ref{r:EMKNScaling} \\
        $M$ & \# of Fourier nodes to estimate $D(L,N)$ & $\gtrsim-\log \varepsilon$ & Prop.~\ref{p:EMDerivativeError}\\
        $M^\prime$ & \# of Fourier nodes to estimate $I(N)$ & $\gtrsim
        -\log \varepsilon / (\log N -\log \nu_A)$ &
        Prop.~\ref{p:EMIntegralError} \\
		\hline
	\end{tabular}
\label{tab:Scalings}
\end{table}

\begin{theorem}
    \label{t:PointwiseEvaluation}
    Given $\varepsilon > 0$ there exist (finite) complex sequences $c_m, c_m^\prime$ and $z_m,
    z_m^\prime$, natural numbers $N,L,M,M^\prime = O(\log \varepsilon)$ and a constant $C_A$ such that 
    for every $\phi \in \Hd[2,2]{r_A}$, every $(x,y) \in [-1,1]^2$ and $s > \frac12$,	
    \begin{multline} 
       \Bigg| (\mathcal{A}_s \phi)(x,y) - \sum_{\pm}\left( \sum_{n=0}^{N-1}
    J_n^\pm(x,y)^s \phi(G_n^\pm(x,y)) + \tfrac{1}{2}J_N^\pm(x,y)^s
    \phi(G_N^\pm(x,y))\right.  \\ 
    + \frac{1}M \sum_{m=1}^M c_m J_{z_m}^\pm(x,y)^s \phi(G_{z_m}^\pm(x,y))
   \left. 
   + \frac{1}{ M^\prime} \sum_{m=0}^{ M^\prime-1}{c_m^\prime} \left(n^2 J_{
  z_m^\prime}^\pm(x,y)\right)^s \phi(G_{z_m^\prime}^\pm(x,y)) \right) \Bigg|  \\ 
  \qquad\qquad\qquad\qquad  \leq \| \phi \|_{\Hd[2,2]{r_A}}C_{\!A}^s\cdot \varepsilon. 
    \label{eq:Theorem2}
	\end{multline}
\end{theorem}

\begin{proof}
Our aim is to apply the Euler--Maclaurin formula together with our estimates
    on correction terms to the Apollonian transfer operator~\eqref{eq:ApTrOp}, which we recall here: 
	\begin{equation*} 
        (\mathcal A_s \varphi)(x,y) = \sum_{n=0}^\infty J^+_n(x,y)^s \varphi(G^+_n(x,y))
        +  J^-_n(x,y)^s \varphi(G^-_n(x,y))
    \end{equation*} 
    To this end, let~$(x,y)$ be fixed. We introduce
        $$
        \psi(n) := J_n^+(x,y)^s \varphi(G_n^+(x,y)) + J_n^-(x,y)^s
        \varphi(G_n^-(x,y)).
        $$
    Let us also introduce the parameters $N, L, M$, and $M^\prime$, to be fixed later, but so that $2L-1 < 2e\pi
    (N-\nu_A)$ and as prescribed by Table~\ref{tab:Scalings}.
    Then the Euler--Maclaurin formula~\eqref{eq:EMFormula} gives 
    $$
    (\mathcal A_s \varphi)(x,y) = \sum_{n=0}^N \psi(n) + \frac12 \psi(N) +
    I_\psi(N) + D_\psi(L,N) + R_\psi(L,N). 
    $$
	We want to apply the results from Sections
    \ref{ss:EMRemainder}--\ref{ss:EMIntegral} to functions $\psi$ and
	\begin{align*} 
	\widetilde \psi(n^{-1}) &:= n^{2s} \psi(n) = \left(n^2 J_n^+(x,y)\right)^s
    \varphi(G_n^+(x,y)) + \left(n^2 J_n^-(x,y)\right)^s \varphi(G_n^-(x,y))
      \end{align*} 
	It follows from Proposition~\ref{p:APrioriConstants}\ref{res:Large} that for all complex $|n| > \nu_A$,
	\[ | \widetilde \psi(n^{-1}) | \leq 2 \cdot 6.8^s \| \varphi
    \|_{\Hd[2,2]{r_A}} \]
	and hence introducing $C_A: = 6.8/\nu_A^{2}$, we conclude
	\[ | \psi(n) | = |\widetilde \psi(n^{-1}) n^{-2s}| \leq  2 \nu_A^{-2s} \cdot
    6.8^s = 2 \| \varphi \|_{\Hd[2,2]{r_A}} C_A^s.\] 

    Next, we need to show that $\widetilde\psi(n^{-1})$ is analytic for $|n| > \nu_A$.
    Proposition~\ref{p:APrioriConstants}\ref{res:Large} gives that
    $G^\pm_n(x,y)$ belongs to 
    the domain of analyticity of $\phi$ for all $|n| > \nu_A$, and
    Proposition~\ref{p:AnalyticAtInfinity} gives that $(n^{-2} J_n^\pm(x,y))^s,
    G^\pm_n(x,y)$ are analytic at $n=\infty$. Hence, it only remains to show that
    $n \mapsto G^\pm_n(x,y), J_n^\pm(x,y)$ are analytic for $|n| > \nu_A$.
    Since the two functions are rational maps of $n$, they
    are analytic whenever they are bounded, and the latter is already
    established above. 
    
    Now, we can apply Propositions \ref{p:EMRemainderBound},
    \ref{p:EMDerivativeError}, \ref{p:EMIntegralError} and collect
    approximation errors together. 
    Applying Proposition~\ref{p:EMRemainderBound} first to estimate the
    remainder term~$R_\psi(L,N)$
    with $\nu = \nu_A$ and $C = C_A $, we get 
	\[R_\psi(L,N) \leq     \frac{ 2 C_A^s (2L+1)!}{L (2\pi)^{2L+1}(N-\nu_A)^{2L}}\| \varphi \|_{\Hd[2,2]{r_A}} . \] 
    Then we choose~$z_m$ and $c_m \in \C$ as specified by~\eqref{eq:zmdef}
    and~\eqref{eq:cmdef} respectively in Proposition~\ref{p:EMDerivativeError}
    and estimate the difference 
    $$
    \left| D(L,N) - \frac{1}M \sum_{m=1}^M c_m J_{z_m}^\pm(x,y)^s
    \varphi(G_{z_m}^\pm(x,y)) \right| \leq  \frac{\pi^2 e}{6}
    \frac{N-\nu_A}{\left(\tfrac{2\pi e (N-\nu_A)}{2L-1}\right)^2 - 1} \frac{
    2 C_A^s}{e^M - 1}\| \varphi \|_{\Hd[2,2]{r_A}} . 
    $$
	Finally, with $c_m^\prime$, and $z_m^\prime$ as given by~\eqref{eq:cmpdef}
    and~\eqref{eq:zmpdef} respectively we use
    Proposition~\ref{p:EMIntegralError} to bound the error 
  $$
  \left| I_\psi(N) - \frac{1}{M^\prime} \sum_{m=0}^{ M^\prime-1} c_m^\prime \left(n^2
  J_{z_m^\prime}^\pm(x,y)\right)^s \varphi(G_{z_m^\prime}^\pm(x,y)) \right|
  \leq \frac{2C_A^s \nu_A^{2s} N^{1-2s}}{(2s-1)(1 -
         \tfrac{\nu_A}{N})} \frac{1}{(\tfrac{N}{\nu_A})^{M^\prime} - 1}
         \| \varphi \|_{\Hd[2,2]{r_A}} .
 $$
Choosing $M,M',L,N$ just large enough with the appropriate relationships between
each other (see Table \ref{tab:Scalings}) ensures that the sum of the right hand
sides is no bigger than~$\varepsilon$.
\end{proof}

\begin{remark}
Note that because the $c_m, z_m, c_m^\prime$, and $z_m^\prime$ are independent of
$\varphi$ and $(x,y)$ they can be calculated once and then be reused over and over
for different choices of function $\varphi \in \mathcal H^{2,2}_{r_A}$ and evaluation
points.
\end{remark}

\subsection{Global estimates}\label{ss:globalestimates}

The following proposition allows us to apply Theorem~\ref{t:Cheby} to estimate
the action of our transfer operator $\mathcal{A}_s$ using Chebyshev--Lagrange
approximation. Recall the error function $E_{2,n}(x) =   16
e^{-(K-1)x}\cdot\frac{1+Kx}{x^{2}}$ defined earlier 
by~\eqref{eq:erFeps}. 
\begin{proposition}\label{p:APrioriConsequences}
	Let $s \in [1.30,1.31]$. Then for every $\psi \in \Hd[2,2]{r_A}$, 
    \[ \| (\mathcal A_s - \P_n \mathcal A_s) \psi \|_{L^\infty} \leq W_A
    \cdot E_{2,n}(R_A) \cdot \| \psi \|_{\Hd[2,2]{r_A}} \]
	where the map-dependent constant $W_A = 3$.
\end{proposition}
\begin{proof}
	We have from Proposition~\ref{p:APrioriConstants}\ref{res:CLJacobianBound}
    \[ \sup_{(x,y) \in \mathcal E^{2,2}_{r_A}} \sum_{n\in \N, \pm}
    |J_n^\pm(x,y)|^s \leq 2 \sum_{n\in \N} \min \Bigl\{
    \frac{3^s}{4^s(n+1)^s}, \frac{36^s}{n^{2s}} \Bigr\}. \] 
	This function is decreasing so has an integral bound
	\[ 2\sum_{n=1}^\infty \min \left\{ (0.75)^{1.31} (n+1)^{-1.30}, 36^{1.31}
    n^{-2\cdot 1.30} \right\} \leq 3 =: W_A. \]
	
	Together with Proposition~\ref{p:AnalyticAtInfinity} we conclude  
    that $\mathcal A_s$ satisfies the assumptions of Theorem~\ref{t:Cheby}
    for $R=R_A, r=r_A, \rho = 0$, and the result follows.
\end{proof}
	It only remains to show that the Apollonian transfer operator satisfies
    hypotheses of Theorem~\ref{t:PositiveOperator2}. Namely, we need to show
    that for some $D^\pm$ the double inequality~\eqref{eq:LinearResponse}
    holds true for $\L_s = \mathcal A_s$. It bounds the pointwise response 
    of the action of $\mathcal A_s$ to changing $s$. 
    We shall do it by partitioning $[-1,1]^2 \times \{1,2,\ldots,30\}$ 
    and employing interval arithmetic and applying integral-based estimate for
    larger~$n$. 

\begin{proposition}\label{p:APrioriDPlusMinus}
	Let $s \in [1.30,1.31]$.
	If $D^+_A = 0.59$, $D^-_A = 3.3$, then for every $\psi>0$ and every $(x,y) \in
    [-1,1]^2$,
\[ \bigl[\tfrac{\d}{\d s} \mathcal A_s \psi\bigr] (x,y)= \sum_{n\in \N, \pm}
\log |J_n^\pm(x,y)|\cdot |J_n^\pm(x,y)|^s \psi(G_n^\pm(x,y)) \in [-D_A^-\sup \psi,
-D_A^+ \inf \psi]. \] 

\end{proposition} 
\begin{proof}
	Since 
	\[ \bigl[\tfrac{\d}{\d s} \mathcal A_s \psi \bigr](x,y) = \sum_{n\in \N,
    \pm} \log |J_n^\pm(x,y)|\cdot |J_n^\pm(x,y)|^s \psi(G_n^\pm(x,y)) \] 
	This means we need to find $D^\pm_A$ such that 
	\[ \sum_{n\in \N, \pm} \log |J_n^\pm(x,y)| \cdot |J_n^\pm(x,y)|^s \in [-D^-_A, -D^+_A]. \]
	We achieve this for $n \le 30$ by dividing $[-1,1]^2$ into boxes
    and applying interval bounds. For the remainder, we obtain uniform upper and
    lower bounds on $n^{-2} J_n^{\pm}(x,y)$ for $n \geq 30$ and then obtain an
    upper bound on the sum via the integral. This is done in the Jupyter notebook {\tt APriori.ipynb}. 
\end{proof}

\section{Computer-assisted proof of the main theorem}
\label{s:algorithm}

Let us now describe the algorithm for the computer program used to obtain the bounds on the Apollonian gasket's Hausdorff
dimension that we claim in Theorem~\ref{t:main}. It has two stages: a non-rigorous guess of the
dimension is found, and this guess is then rigorously validated. We begin by choosing a desired
precision~$\varepsilon$ and computing parameters $K, L, N, M,
M^\prime$ as specified in Table~\ref{tab:Scalings}.

The first step is an iterative search of a parameter value~$s_*$
such that the maximal eigenvalue of $\mathcal{P}_K \mathcal A_{s_*}$ is
sufficiently close to~$1$.  

The search is initialised with two guesses $s_1, s_2$ close to $\dim_H \apack$. For each guess, we evaluate the action of our transfer
operator~$\mathcal A_{s_1,s_2}$ on the Chebyshev--Lagrange basis at the Chebyshev nodes. In other words, for $s_1,s_2$ in turn
we construct a (complete) matrix representation of $\mathcal{P}_K \mathcal
A_{s_1,s_2}$ in the
 Chebyshev--Lagrange basis. We then compute the leading eigenvalues of this matrix as specified in
 Algorithm~\ref{alg:nonrigeigval}. By Theorem~\ref{t:Cheby} we know that our
 matrix representation is a good 
approximation of the true transfer operator, and so its leading eigenvalue will be a
reasonable approximation of the true leading eigenvalue of~$\mathcal A_s$. 

The search proceeds with the secant method applied to the function from
Algorithm~\ref{alg:nonrigeigval} as specified in Algorithm~\ref{alg:nonrig}. The
aim is to find a value of~$s$ for which the leading eigenvalue is sufficiently
close to~$1$. We do not need to explicitly estimate any errors here, and so
work in extended floating point arithmetic (such that the machine precision is around $\epsilon$). Then both the value~$s_*$ and the corresponding
eigenvector, the polynomial~$\varphi_*$, represented by its values at the Chebyshev nodes are
passed to the second step for verification. 

The second step is a rigorous verification using the min-max method. Namely, we apply the
approximating operator~$\P_K\mathcal A_s$ to the guessed eigenvector~$\varphi_*$, keeping
track of all errors via interval arithmetic. In order to invoke
Theorem~\ref{t:PositiveOperator2}, we only need to obtain bounds on the discrepancy $\varepsilon^-<|\P_K\mathcal
A_s\varphi_*-\varphi_*|<\varepsilon^+$, and to compute necessary constants using
Propositions~\ref{p:APrioriConsequences} and~\ref{p:APrioriDPlusMinus} with
$\psi = \varphi_*$ and $s = s_*$. The verification step is realised in Algorithm~\ref{alg:rig}.    

After the preliminary discussion, we are ready to justify the accuracy of the
dimension value claimed. 


\begin{proof}[Proof of Theorem~\ref{t:main}]
	The bounds on the dimension~$s_-\le \dim_H \apack \le s_+ $ is the output of
    Algorithm~\ref{alg:rig}, with all computations done
    using the interval arithmetic so that
    any round-off errors are taken into account at all times. 
 
    Let $s_*$ and $\varphi_*$ be an educated guess of the dimension value and
    a polynomial eigenfunction, represented by its values at the Chebyshev nodes.
	Theorem~\ref{t:PointwiseEvaluation} allows us to evaluate $(\mathcal
    A_{s_*} \varphi_*)(x,y)$ for given values $x,y$ with arbitrary precision using the
    Euler--Maclaurin formula.  
    Hence we can calculate the polynomial $\P_K \mathcal A_{s_*} \varphi_*$
    numerically using evaluations of $\mathcal A_{s_*} \varphi_*$ at
    Chebyshev--Lagrange nodes. 
    
    Then Proposition~\ref{p:APrioriConsequences} allows us
    to obtain an estimate on $\|\varphi_*\|_{\Hd[2,2]{r_A}}$ and thus bound the
    norm of the difference $\|\P_n \mathcal A_{s_*} \varphi_* -
    \mathcal A_{s_*} \varphi_*\|_{L^\infty }$
    in terms of~$\|\varphi_*\|_{\Hd[2,2]{r_A}}$.
	At the same time, using Proposition~\ref{p:APrioriDPlusMinus} we compute
    necessary constants in order to apply
    Theorem~\ref{t:PositiveOperator2} to operator $\mathcal A_{s_*}$. This gives us an interval $[s_+, s_-] \subseteq [1.30,1.31]$
    containing an~$s$ such that $\rho(\mathcal A_s) = 1$.
    Proposition~\ref{p:SpecRadIsDimension} implies that this value of $s$ is
    unique, and is the Hausdorff dimension of the Apollonian circle packing. 
\end{proof}

\subsection{Computational expense}

We will briefly comment on the computational power used by this algorithm. To
evaluate $\mathcal A_s \psi$ at a given point takes around $O(N + L + M +
M^\prime) = O(\log \varepsilon^{-1})$ evaluations of $\psi$ and other
functions. Given that $K = O(\log \varepsilon^{-1})$, the construction of
a $K^2 \times K^2$ transfer operator in the nonrigorous step therefore takes
approximately $O((\log \varepsilon^{-1})^5)$ evaluations; so too does the
evaluation of $\mathcal A_s$ applied to a sum of $K^2$ Chebyshev--Lagrange polynomials at $K^2$
points in the rigorous step. The secant method requires $O(\log \log
\varepsilon^{-1})$ steps to converge. The time taken to perform arithmetic
operations using {\tt BigFloat}s (the default extended-precision floating point
arithmetic in Julia) appears empirically to be around $O((\log
\varepsilon^{-1})^{1.5})$ for precision $\varepsilon$. 

This suggests that the computational power required to obtain a rigorous
Hausdorff dimension estimate of precision $\varepsilon$ is\footnote{Here and
below by $o_\varepsilon(1)$ we understand a quantity that $\to 0$ as $\varepsilon
\to 0$.}  $O((\log
\varepsilon^{-1})^{6.5+o_\varepsilon(1)})$: which is to say, doubling the number of
digits in the estimate requires around $100$ times more computing power. Our
algorithms are of course extremely parallelisable. 

We note that we were able to make some efficiency gains on this by using other
structure of the problem. For example, the transfer operator preserves $y$-even
functions, implying by positivity that the leading eigenfunction is also
$y$-even: as a result, we restricted our discretisation to $y$-even
polynomials, which provided a fourfold saving in computational expense. 

In our implementation we ran our algorithm with $\varepsilon = 2^{-447}
\approx 3 \times 10^{-135}$ on $30$ cores of a research server. This took around
90 hours in total. 
The programming code is written in Julia and is available 
at~\url{https://github.com/wormell/apollonian}.  


Below we provide pseudocode for the key functions of the implementation.

\bigskip

\begin{algorithm}
	\caption{Non-rigorous estimate of the leading eigenvalue and eigenvector
    of~$\mathcal A_s$ for given~$s$.}\label{alg:nonrigeigval}
	
	\begin{algorithmic}[1]

		\Function{NonRigorousEigenvalue}{$s$, $K$, $N$, $L$, $M$, $M^\prime$}
		
        \Require{$s \in [1.30,1.31]$: the current estimate for the dimension;
                $K$, $N$, $L$, $M$, $M^\prime$ as in Table~\ref{tab:Scalings}}
	
        \Let{$x_{\vec j,K}$}{Chebyshev nodes~\eqref{eq:ChebyshevNodes}} 
        \Let{$\ell_{\vec k,K}$}{Chebyshev--Lagrange polynomials~\eqref{eq:LagrangePoly}}

		\For{$\vec k \in \{1,2,3,\ldots, K\}^2$}
				\For{$\vec j \in \{1,2,3,\ldots, K\}^2$}
		\Let{TransferMatrix$_{\vec j, \vec  k}$}
        {\Call{PointwiseEstimate}}{$s, x_{\vec j,K}, \mathcal A_s\ell_{\vec k,K},N,L,M, M^\prime$} 
				
		\Comment{Approximation to $(\mathcal A_s \ell_{\vec k,K})(x_{\vec j,K})$
         according to Theorem~\ref{t:PointwiseEvaluation} }
		\EndFor
		\EndFor 
		
		\Let{$\lambda, \vec \varphi$}{\Call{ComputeLeadingEigs}{TransferMatrix}}
        

        \Let{$\Phi$} {Evaluate eigenfunction reconstructed from $\vec \varphi$
        at $K \times K$ Chebyshev nodes}

        \State\Return{$\Phi$, $\lambda$}
		\EndFunction
	\end{algorithmic}
\end{algorithm}
\begin{algorithm}
\caption{Non-rigorous estimate of the spectral radius and accompanying eigenfunction.}\label{alg:nonrig}
\begin{algorithmic}[1]
		\Function{NonRigorousDimensionEstimate}{$\varepsilon,K,N,L,M,\tilde M$}
		
		\Require{
			$\varepsilon<10^{-10}$: the desired precision, $K,N,L,M, M^\prime$
             per Table~\ref{tab:Scalings}.}

		\Let{$s_0$}{$1.30$} 
		\Let{$\lambda_0$, $\Phi_0$}{\Call{NonRigorousEigenvalue}{$s_0$, $K$, $N$, $L$, $M$, $M^\prime$}}
		\Let{$s_1$}{$1.31$} 
		\Let{$\lambda_1$, $\Phi_1$}{\Call{NonRigorousEigenvalue}{$s_1$, $K$, $N$, $L$, $M$, $M^\prime$}}
		\Let{$t$}{$1$}
		
        \algcommentleft{The secant method is used to find~$s$ corresponding to
        $\lambda = 1$; }

        \algcommentleft{the rate of convergence is $O(e^{-c t^\alpha})$,
        where~$\alpha$ is the golden ratio.} 

        \While{$|\lambda_{t-1} - 1|\cdot |\lambda_{t-2}-1| > \varepsilon$}
				\Let{$t$}{$t+1$}
		\Let{$s_t$}{$s_{t-1} - (\lambda_{t-1}-1)\tfrac{s_{t-1}-s_{t-2}}{\lambda_{t-1}-\lambda_{t-2}}$}
		\Let{$\lambda_t$,  $\Phi_t$}{\Call{NonRigorousEigenvalue}{$s_t$, $K$,
        $N$, $L$, $M$, $M^\prime$}}
		\EndWhile
		
		\State\Return{$s_t$, $\Phi_t$}
		\EndFunction
	\end{algorithmic}
\end{algorithm}
\begin{algorithm}
	\caption{Rigorous estimate of the dimension of the Apollonian gasket. }\label{alg:rig} 
			\begin{algorithmic}[1]
				\Function{RigorousDimensionBound}{$\varepsilon$}
		
		\Require{$\varepsilon$ -- the desired accuracy of the estimate.}
		
		\State \Call{SetFloatingPointPrecision}{$\varepsilon$}

		\Let{$K$}{$2 \lceil -\tfrac{\log \varepsilon}{2R_A} \rceil$}
        \Comment{$R_A = 1.4$, as specified in \S\ref{s:aprioribdd}}
		\Let{$N$}{$\lceil \nu_A - \tfrac{ \log \varepsilon}{2\pi}\rceil$}
        \Comment{$\nu_A = 10$, as specified in \S\ref{s:aprioribdd}}
		\Let{$L$}{$\lfloor \tfrac{2(N-\nu_A) - 1}{2}\rfloor$}
		\Let{$M$}{$2L$} 
		\Let{$M^\prime$}{$2\lceil-\tfrac{\log\varepsilon}{2\log (N/\nu_A)}\rceil$}
		\Let{$s$,
        $\Phi$}{\Call{NonRigorousDimensionEstimate}{$\varepsilon$, $K$, $N$, $L$, $M$, $M^\prime$}}
	
        \Let{$x_{\vec j,K}$}{Chebyshev nodes \eqref{eq:ChebyshevNodes}} 
        \Let{$\ell_{\vec k,K}$}{Chebyshev--Lagrange polynomials \eqref{eq:LagrangePoly}}

\State \Call{UseIntervalArithmetic}{\phantom{}}

	    \For{$\vec j \in \{1,2,3,\ldots, K\}^2$}
		\Let{$\vec\chi_{\vec j}$}
        {\Call{PointwiseEstimate}}{$s, x_{\vec j,K}, \Phi_{\vec k}\ell_{\vec k,K},N,L,M, M^\prime$} 

		\EndFor
		
		\Let{$\widehat \varphi$}{\Call{DiscreteChebyshevTransform}{$\Phi$}}

        \Comment{coefficients of the eigenfunction in the basis of Chebyshev
        polynomials on~$[-1,1]^2$}
		
        \Let{VNorm}{$\sum_{\vec k} e^{r_A \|\vec k\|_{\ell^2}} |\Phi_{\vec k}|$}
        \Comment{upper bound on $\| \varphi\|_{\Hd[2,2]{r_A}}$ by Corollary~\ref{cor:|phi|H}.}
        \Let{Erp}{\Call{PointwiseEstimateError}{$\varepsilon$, Vnorm, $N$, $L$,
        $M$, $M^\prime$}}

        \Comment{ right hand side of~\eqref{eq:Theorem2}}
        
        \Let{ApproximationError}{$W_A \cdot$VNorm$\cdot E_{2,K}(R_A) +$ Erp}

                
        \Comment{taking in error estimate from
        Proposition~\ref{p:APrioriConsequences}; $W_A = 3$, as specified in
        \S\ref{ss:globalestimates}}

        \Let{$[\varphi^-,\varphi^+]$}{$\widehat \varphi_{0,0} + (\sum_{ \vec k \ne 0}
        |\widehat \varphi_{\vec k}|)\cdot [-1,1]$}
		
        \Let{$\widehat{er}$}{\Call{DiscreteChebyshevTransform}{$\chi-\Phi$}}
        \Let{$[e^-,e^+]$}{$\widehat{er}_{0,0} + (\sum_{\vec k \neq 0}
        |\widehat{er}_{\vec k}|)\cdot [-1,1]$}
		\Let{$s_-$}{$s - (e^- - \textrm{ApproximationError})/{D^-_A\varphi^+}$}
		\Let{$s_+$}{$s - (e^+ + \textrm{ApproximationError})/{D^+_A\varphi^-}$}

        \Comment{constants $D_A^{\pm}$ are given in Proposition \ref{p:APrioriDPlusMinus}}

        \State\Return{$[s_-,s_+]$} 
        \Comment{Interval of length $\sim \varepsilon^{1-o_{\varepsilon}(1)}$
        containing the Hausdorff dimension of the gasket.}
		\EndFunction
	\end{algorithmic}
\end{algorithm}

\clearpage


\section*{Acknowledgements}
The problem was introduced to the first author by prof Mark Pollicott (University of
Warwick) and without his generous support this paper would have never been completed. 

We are grateful to prof Mariusz Urbanski (North Texas), prof Carlangelo Liverani
(Rome Tor Vergata) and Dr Oscar Bandtlow
(QMUL) for pointing us to the results that led to the proof of Proposition~\ref{p:SpecRadIsDimension}. 

The first author was partly supported by EPSRC grant EP/T001674/1 (PI prof Mark
Pollicott). 
The second author was partly supported by the European Research Council (ERC) under the European
Union Horizon 2020 research and innovation programme (grant agreement Nos 787304
and 833802, PI prof Viviane Baladi)
and by Australian Mathematics Society WIMSIG Cheryl E. Praeger Travel Award. 

The authors would like to thank CIRM Luminy (Marseille, France) and the University of
Warwick (Coventry, UK) for their hospitality. 

Last but not least, we would like to thank the School of Mathematics and Statistics at the University of Sydney for providing access to the
research server where the computations were performed.

\bigskip

\centerline{\rule{70pt}{1pt}}

\bigskip

{\sc Caroline Wormell},  School of Mathematics and Statistics,  University of
Sydney, NSW 2006, Australia; \
e-mail: caroline.wormell@sydney.edu.au

{\sc Polina Vytnova}, School of Mathematics and Physics, University of Surrey,
Guildford, Surrey, GU2 7XH, UK; \
e-mail: P.Vytnova@surrey.ac.uk


\end{document}